\documentclass[letterpaper]{amsart}

\usepackage[all]{xy}                        %

\CompileMatrices                            

\UseTips                                    

\input xypic
\usepackage[bookmarks=true]{hyperref}       

\usepackage{amssymb,latexsym,amsmath,amscd}
\usepackage{xspace}
\usepackage{enumerate}
\usepackage{graphicx}

\reversemarginpar

\theoremstyle{plain}
\newtheorem{theorem}{Theorem}[section]
\newtheorem*{theorem*}{Theorem}
\newtheorem{proposition}[theorem]{Proposition}

\newtheorem{lemma}[theorem]{Lemma}

\theoremstyle{definition}
\newtheorem{definition}[theorem]{Definition}

\newtheorem{notation}[theorem]{Notation}

\newtheorem{remark}[theorem]{Remark}

\DeclareMathOperator{\Res}{Res}

\newcommand{\enm}[1]{\ensuremath{#1}}          %

\newcommand{\cal}[1]{\mathcal{#1}}

\newcommand{\NN}{\enm{\mathbb{N}}}

\renewcommand{\AA}{\enm{\mathbb{A}}}
\newcommand{\PP}{\enm{\mathbb{P}}}

\newcommand{\Ii}{\enm{\cal{I}}}

\newcommand{\Oo}{\enm{\cal{O}}}

\newcommand{\Ss}{\enm{\cal{S}}}

\newcommand{\Y}{\mathbb{P}^{n_1}\times \mathbb{P}^{n_2}}

\newcommand{\SV}{S_{d_1, \ldots , d_k}(V_1, \ldots , V_k)}
\newcommand{\SVd}{S_{d_1, d_2}(V_1,  V_2)}

\renewcommand{\phi}{\varphi}
\renewcommand{\theta}{\vartheta}
\renewcommand{\epsilon}{\varepsilon}

\begin{document}

\title[Unique decomposition for bi-homogeneous polynomials]{A uniqueness result on the decompositions of a bi-homogeneous polynomial}

\author[E. Ballico]{Edoardo Ballico}
\address[Edoardo Ballico]{Dipartimento di Matematica,  Univ. Trento, Italy}
\email{edoardo.ballico@unitn.it }

\author[A. Bernardi]{Alessandra Bernardi}
\address[Alessandra Bernardi]{Dipartimento di Matematica, Univ.  Trento,  Italy}
\email{alessandra.bernardi@unitn.it}

\maketitle


\begin{abstract}
In the first part of this paper we  give a precise description of all the minimal decompositions of any bi-homogeneous polynomial  $p$ (i.e. a partially symmetric tensor of $S^{d_1}V_1\otimes S^{d_2}V_2$ where $V_1,V_2$ are two complex, finite dimensional vector spaces) if its rank with respect to the Segre-Veronese variety $S_{d_1,d_2}(V_1,V_2)$  is at most $\min \{d_1,d_2\}$. Such a polynomial may not have a unique minimal decomposition as $p=\sum_{i=1}^r\lambda_i p_i$ with $p_i\in S_{d_1,d_2}(V_1,V_2)$ and $\lambda_i$ coefficients, but we can show that  there exist unique $p_1, \ldots , p_{r'}$, $p_{1}', \ldots , p_{r''}'\in S_{d_1,d_2}(V_1,V_2) $, two unique linear forms $l\in V_1^*$, $l'\in V_2^*$, and two unique bivariate polynomials $q\in S^{d_2}V_2^*$ and $q'\in S^{d_1}V_1^*$ such that either $p=\sum_{i=1}^{r'} \lambda_i p_i+l^{d_1}q $ or $ p= \sum_{i=1}^{r''}\lambda'_i p_i'+l'^{d_2}q'$, ($\lambda_i, \lambda'_i$ being appropriate coefficients). 

In the second part of the paper we focus on the tangential variety of the Segre-Veronese varieties. We compute the rank of their tensors (that is valid also in the case of Segre-Veronese of more factors) and we describe the structure of the decompositions of the elements in the tangential variety of the two-factors Segre-Veronese varieties.
\end{abstract}


\section*{Introduction}

Let $V_1, V_2$ be vector spaces of dimension $n_i+1$ for $i=1,2$ defined over an algebraically closed vector field $K$ of characteristic zero. The space $S^{d_1}V_1\otimes S^{d_2}V_2$ is the space of partially symmetric tensors of type $T_1\otimes T_2 $ where $T_i\in S^{d_i}V_i$ is a completely symmetric tensor of order $d_i$ for $i=1, 2$. Since $S^{d_i}V_i$ can be interpreted also as the space of homogeneous polynomials of degree $d_i$ in the set of variables $\{x_{i,0}, \ldots , x_{i,k}\}$ defined over $K$, i.e. $S^{d_i}V_i^*\simeq K[x_{i,0}, \ldots , x_{i,n_i}]_{d_i}$ for $i=1, 2$, then the space  $S^{d_1}V_1^*\otimes S^{d_2}V_2^*$ represents also bi-homogeneous polynomials of type $p=p_1p_2$ with $p_i\in K[x_{i,0}, \ldots , x_{i,n_i}]_{d_i}$ for $i=1,2$.

The embedding of $\mathbb{P}(V_1)\times \mathbb{P}(V_2)$ into $\mathbb{P}(S^{d_1}V_1\otimes  S^{d_2}V_2)$ induced by the complete linear system $|\mathcal{O}_{\mathbb{P}(V_1)\times  \mathbb{P}(V_2)}(d_1,  d_2)|$ is the so called \emph{two factors Segre-Veronese  variety} and it is denoted by $\SVd$. It can be viewed as the variety parameterizing projective classes of partially symmetric tensors that can be written as:
$$T=v_1^{\otimes d_1}\otimes v_2^{\otimes d_2}$$
with $v_i\in V_i$ for $i=1, 2$.
In terms of multi-homogeneous polynomials,  $\SVd$ can be interpreted as the variety parameterizing projective classes of bi-homogeneous polynomial of type
$$\label{rank1p}p=l_1^{d_1} l_2^{d_2}$$
where $l_i$ are linear  forms in $S^1V_i^*\simeq K[x_{i,0}, \ldots , x_{i,n_i}]_1$, $i=1,2$.

We will say that an element of the Segre-Veronese variety has \emph{rank $1$}. The minimum integer $r$ such that  a bi-homogeneous polynomial $p$  (a two factors partially symmetric tensor $T$) can be written as a linear combination of $r$ rank 1 bi-homogeneous polynomials (two factors partially symmetric tensors) is called the \emph{rank} of $p$ and it is denoted by $r(p)$ (or $r(T)$ respectively). By an abuse of notation we will say that such an $r$ is also the \emph{rank} of the projective class $[p]$ of $p$  (the projective class $[T]$ of $T$ respectively).

\medskip

From now on we will indicate with $p$ both the bi-homogeneous polynomial and the corresponding  partially symmetric tensor.

\medskip

One of the main problems of the fieldwork on a minimal decomposition of a polynomial or of a tensor is the knowledge of its possible uniqueness or \emph{identifiability}. Many branches of pure and applied mathematics are nowadays very active in this field, see for example \cite{BC, amr, cov, co, b1, bc, bcv, ddl1, ddl2, ddl3, bbcc, bb3}.

Suppose that $\mathcal{W}\subset \PP^r$ is a non-degenerate reduced and irreducible projective variety and that a point $[p]\in \mathcal{W}$ lies on a $r$-secant space  $H\simeq \mathbb{P}^{r-1}$ to $\mathcal{W}$  and not on any $\mathbb{P}^{r-2}$ that is $(r-1)$-secant. 
A very general fact on the uniqueness of minimal decomposition is the following one.

\begin{definition}\label{rho} Let
$\rho ' (\mathcal{W})$ be the maximal integer $t$ such that any subset of $\mathcal{W}$ with cardinality $t$ is linearly independent.
\end{definition}

\noindent \textbf{General fact:}
If  $2r\le \rho '(\mathcal{W})$, then $H$ is the only one $r$-secant space to $\mathcal{W}$ containing $[p]$ (crf.  \cite[Theorem 1.18]{bgl}).

\smallskip

This fact, translated in terms of bi-homogeneous polynomials (or two factors  partially symmetric tensors), means that if $p\in S^{d_1}V_1^*\otimes S^{d_2}V_2^*$  is such that 
$$r(p)\leq \rho'(\SVd)$$ 
then $p$ has a unique minimal decomposition as 
\begin{equation}\label{decompPoly}p=\sum_{i}^r \lambda_i l_i^{d_1}l_i'^{d_2}\end{equation} 
with $l_i\in V_1^*, l'_i\in V_2^*$, $\lambda_i\in K$, $i=1, \ldots, r$.
\\
In terms of two factors partially symmetric tensors it means that 
\begin{equation}\label{decompTensor}p=\sum_{i=1}^r \lambda_i v_i^{\otimes d_1}\otimes v_i'^{\otimes d_2}\end{equation} 
with $v_i\in V_1, v_i'\in V_2$, $\lambda_i\in K$, $i=1, \ldots, r$.

If $X$ is the Segre-Veronese variety of $k$ factors (i.e. $X=\SV$  is the embedding of $\mathbb{P}(V_1)\times \cdots \times \mathbb{P}(V_k)$ into $\mathbb{P}(S^{d_1}V_1\otimes \cdots \otimes  S^{d_k}V_k)$ induced by the complete linear system $|\mathcal{O}_{\mathbb{P}(V_1)\times  \cdots \times \mathbb{P}(V_k)}(d_1,  \ldots , d_k)|$), then  we have that
$\rho '(X) = 1 +\min _{1\le i \le k} \{d_i\}$ (in absence of a standard reference for this quite obvious fact,  for sake of completeness, we give the proof in Lemma \ref{suck} at the end of Section \ref{SectionTwoFactors}).
Unfortunately this integer is quite low, but in the case of bi-homogeneous polynomials where 
$$\rho'(\SVd)=1+\min\{d_1,d_2\},$$
we may get a stronger uniqueness result (roughly by a factor $2$).
In the main result of this paper that is Theorem \ref{ee1.1} we show the exact structure of the unique minimal decomposition of a bi-homogeneous polynomial $p$ (order 2 partially symmetric tensor) with $r(p)\leq \rho'(\SVd)$. Moreover we can also prove that the same decomposition's structure holds for a bigger class of bi-homogeneous polynomials, namely it holds for any $p\in S^{d_1}V_1^*\otimes S^{d_2}V_2$ with $2r({p}) \le 1+d_1+d_2$ and $|d_1-d_2| \le 2$ (the case where $p$ is actually a homogeneous polynomial in only one set of variables is slightly different, we treat it separately, cfr. (\ref{third}) in Theorem \ref{ee1.1}, and we don't describe it here in the Introduction). In this last case the decomposition as sum of elements in $\SVd$ won't be unique anymore, but we have another kind of uniqueness. 
\\
In order to facilitate the reading of the first two items of Theorem \ref{ee1.1}, let us be more explicit here. In both cases, i.e. if either $2r({p}) \le 1+d_1+d_2$ and $|d_1-d_2| \le 2$ (except if $p$ is a homogeneous polynomial where the situation will be anyway explicitly described in (\ref{third}) of Theorem \ref{ee1.1} and Remark \ref{lll1}) or if $r({p}) \le \min \{d_1,d_2\}$ we show that there exist: 
\begin{itemize}
\item Unique $p_1, \ldots , p_{r'}, p_{1}', \ldots , p_{r''}'\in S_{d_1,d_2}(V_1,V_2) $, 
\item Two unique linear forms $l\in V_1^*$, $l'\in V_2^*$,
\item Two unique spaces of bi-variate linear forms $W_1^*=K[m,n]_1\subset V_1^*$, $W_2^*=K[m',n']_1\subset V_2^*$, with $m,n\in V_1^*,m',n'\in V_2^*$ linear forms, and
\item The following bivariate polynomials $q_1, \ldots , q_s \in S^{d_2}W_2^*$, $q_1', \ldots , q_{s'}' \in  S^{d_1}W_1^*$ 
\end{itemize}
such that either
$$p=\sum_{i=1}^{r'} \lambda_i p_i+l^{d_1}\cdot\left( \sum_{i=1}^s\gamma_iq_i \right),$$
or
$$p=\sum_{i=1}^{r''}\lambda'_i p_i'+\left(\sum_{i=1}^{s'}\gamma'_i q'_i\right)\cdot l'^{d_2}$$
with $\lambda_i, \lambda'_j, \gamma_k , \gamma'_h\in K$.
\\
If we are in case in which the bi-homogeneous polynomial  $p$ has at least two different decompositions (i.e. the $q_i$'s and the $q_i'$'s are not unique), then $2r({p}) >\min \{d_1,d_2\}$ and there are infinitely many choices for $\{q_1, \ldots , q_s\}\subset S^{d_2}W_2^*$ and infinitely many choices for $\{q_1', \ldots , q_{s'}'\}\subset S^{d_1}W_1^*$, but the forms 
 $q=\sum_{i=1}^s\gamma_iq_i\in S^{d_2}W_2^*$ and $q'=\sum_{i=1}^{s'}\gamma_iq_i'\in S^{d_1}W_1^*$ will be unique (more precisely, there are infinitely many choices if and only if there are at least two choices and this is the case if and only if either $s > \lfloor (d_1+1)/2\rfloor$ or $s > \lfloor (d_2+2)/2\rfloor$). Therefore, in this last case, we will have that either
$$p=\sum_{i=1}^{r'} \lambda_ip_i+l^{d_1}q,$$
or
$$p=\sum_{i=1}^{r''}\lambda'_ip_i'+q'l'^{d_2}$$
and all the forms appearing in the decomposition will be unique. In this sense we can speak of ``~unique decomposition~" of the bi-homogeneous polynomial $p$. Knowing either $q$ or $q'$, the finding of  $\{q_1, \ldots , q_s\}\subset S^{d_2}W_2^*$  or of  $\{q_1', \ldots , q_{s'}'\}\subset S^{d_1}W_1^*$ is assured by the classical study of bivariate polynomials
with rank bigger than their border rank due to the celebrated Sylvester's theorem (cfr. \cite{b}, \cite{cs}, \cite[\S 1.3]{ik}, \cite{l} and \cite{bgi,bcmt} for algorithmic computation of the solutions). 

\smallskip

We can rephrase 
all this in terms of two-factors partially-symmetric tensors. Let $p\in S^{d_1}V_1\otimes S^{d_2}V_2$.
If either $2r({p}) \le 1+d_1+d_2$ and $|d_1-d_2| \le 2$ (except again if $p$ is a completely symmetric tensor where the situation will be anyway explicitly described in (\ref{third}) of Theorem \ref{ee1.1} and Remark \ref{lll1}) or if $r({p}) \le \min \{d_1,d_2\}$ we show that there exist: 
\begin{itemize}
\item Unique vectors $v_{j,1}, \ldots , v_{j,r'},v_{j,1}', \ldots , v_{j,r''}'\in V_j$, for $j=1,2$,
\item Two unique vectors $u\in V_1$, $u'\in V_2$,
\item Two unique lines $W_1\subset V_1$, $W_2\subset V_2$ and
\item The following vectors $w_1, \ldots , w_s \in W_2$, $w_1', \ldots , w_{s'}' \in W_1$ 
\end{itemize}
such that either
$$p=\sum_{i=1}^{r'} \lambda_iv_{1,i}^{\otimes d_1}\otimes v_{2,i}^{\otimes d_2}+u^{\otimes d_1}\otimes\left(\sum_{i=1}^s \gamma_iw_i^{\otimes d_2}\right),$$
or
$$p=\sum_{i=1}^{r''}\lambda'_i v_{1,i}'^{\otimes d_1}\otimes v_{2,i}'^{\otimes d_2}+\left(\sum_{i=1}^{s'}\gamma'_iw_i^{\otimes d_1}\right)\otimes u'^{\otimes d_2}$$
with $\lambda_i, \lambda'_j, \gamma_k, \gamma'_h\in K$.
\\
If $p$ is a  two-factors partially-symmetric tensor without unique decomposition (i.e. $w_i$'s and $w_i'$'s are not unique), then
$r({p}) > \min \{d_1,d_2\}$, then there are infinitely many choices for  $\{w_1, \ldots , w_s\}\subset W_2$ and infinitely many choices for $\{w_1', \ldots , w_{s'}'\}\subset W_1$, but the tensors
 $w=\sum_{i=1}^s\gamma_iw_i^{\otimes d_2}\in S^{d_2}W_2$ and $w'=\sum_{i=1}^{s'}\gamma_iw_i'^{\otimes d_1}\in S^{d_1}W_1$ will be unique (more precisely, there are infinitely many choices if and only if there are at least two choices and this is the case if and only if either $s > \lfloor (d_1+1)/2\rfloor$ or $s > \lfloor (d_2+2)/2\rfloor$).  Therefore, in this last case, we will have that either
$$p=\sum_{i=1}^{r'} \lambda_iv_{1,i}^{\otimes d_1}\otimes v_{2,i}^{\otimes d_2}+u^{\otimes d_1}\otimes w,$$
or
$$p= \sum_{i=1}^{r''}\lambda'_iv_{1,i}'^{\otimes d_1}\otimes v_{2,i}'^{\otimes d_2}+w' \otimes u'^{\otimes d_2}$$
and all the tensors appearing in the decomposition will be unique. In this sense we can speak of ``~unique decomposition~" of the tensor $p$. As above,  knowing either $w$ or $w'$, the finding of  $\{w_1, \ldots , w_s\}\subset W_2^*$  or of  $\{w_1', \ldots , w_{s'}'\}\subset W_1^*$ is assured by the classical study of bivariate polynomials
with rank bigger than their border rank due to the celebrated Sylvester's theorem (cfr. \cite{b}, \cite{cs}, \cite[\S 1.3]{ik}, \cite{l} and \cite{bgi,bcmt} for algorithmic computation of the solutions).

\smallskip

It will be remarkable that the numbers $r'$ and $r''$ and the subspaces $W_1\subset V_1$ and $W_2\subset V_2$ will depend only on $[p]$ and not on the decomposition (this will be the content of Proposition \ref{ee7}).

\bigskip

In the second part of the paper we focus on tangential variety to Segre-Veronese variety. \smallskip

\smallskip

In Section \ref{RankTg} we will consider the Segre-Veronese variety $\SV$ of any number of factors.
\\
We will indicate with $r_{d_1, \ldots , d_k}(p)$ the minimum integer $r$ such that the projective class of the multi-homogeneous polynomial $p\in S^{d_1}V_1^* \otimes \cdots \otimes S^{d_k}V_k^*$ (the partially symmetric tensor $p\in S^{d_1}V_1\otimes \cdots \otimes S^{d_k}V_k$) can be written as a sum of elements in $\SV$. Since in this case there won't be any risk of confusion on the number of factors of the Segre-Veronese variety, by an abuse of notation we will call $r_{d_1, \ldots , d_k}(p)$ the \emph{rank of $p$}. 

In Section \ref{RankTg} we show that the rank of any point $[p]$ in the tangential variety of the Segre-Veronese of $k$-factors $ \tau(\SV)$ is $\sum_{i=1}^h d_i$ if $V_1^*=K[x_{1,0}, \ldots , x_{1,n_1}]_1$, $\ldots ,$ $V_h^*=K[x_{h,0}, \ldots , x_{h,n_h}]_1$ are the minimum sets of variables to which the multi-homogeneous polynomial $p$ actually depends on, $h\leq k$.  In terms of partially symmetric tensors this means that the tensor depends actually on $h\leq k$ factors: $p\in S^{d_1}V_1\otimes \cdots \otimes S^{d_h}V_h\subset S^{d_1}V_1\otimes \cdots \otimes S^{d_k}V_k$.

\smallskip

Finally in Section \ref{SectionTangential} we show that, if we keep focusing on the two-factors Segre-Veronese variety, then we are able to use all the mechanism that we have developed in Section \ref{SectionTwoFactors} to describe the structure of the decompositions of the elements in $\tau(\SVd)$. In Theorem \ref{i2} we show that the decomposition of an element $[p]\in \tau(\SVd)$ is always of type 
$$p=l_1^{d_1}\cdot\left( \sum_{i=1}^{r_1}\lambda_im_{i}^{d_2}\right)+ \left(\sum_{i=1}^{r_2}\gamma_in^{d_1}_i \right) \cdot l_2^{d_2}$$
with $r_1+r_2=r(p)$, $m_i\in K[l_1,l_1']_1$, $n_i\in K[l_2,l_2']$, binary linear forms ($l_i,l_i'$ are linear forms in $V_i^*$, $i=1,2$) and $\lambda_j, \gamma_k\in K$. This decomposition has the obvious two ``~exceptions~" of either $r_1=0$ or $r_2=0$ where only one of the two addenda appears in the decomposition.

This can be translated in terms of partially symmetric tensors by saying that any element of the tangential variety of the two factors  Segre-Veronese can be decomposed as
$$p=v_1^{\otimes d_1}\otimes \left( \sum_{i=1}^{r_1}\lambda_iw_{i}^{\otimes d_2}\right)+ \left(\sum_{i=1}^{r_2}\gamma_iu_i^{\otimes d_1}\right)\otimes v_2^{\otimes d_2}$$
where  $w_i\in \langle v_1, v_1'\rangle$ and $u_i\in \langle v_2, v_2'\rangle$ with $v_i,v_i'\in V_i$ for $i=1,2$, and $\lambda_j, \gamma_k\in K$, except if either $r_1=0$ or $r_2=0$ and then only one of the two addenda appears in the decomposition.


\section{A unique decomposition theorem for Segre-Veronese of two factors}\label{SectionTwoFactors}

Now denote by
$$\nu_{d_1,d_2} : \PP(V_1)\times  \PP(V_2) \to \PP(S^{d_1}V_1 \otimes  S^{d_2}V_2)$$
the Segre-Veronese embedding of bi-degree $(d_1,d_2)$  induced by the complete linear system $|\Oo _{ \PP(V_1)\times  \PP(V_2)}(d_1,d_2)|$.

\medskip

Since $\dim V_i=n_i+1$, all along this paper we will use indistinctly the notation $\mathbb{P}(V_i)=\mathbb{P}^{n_i}$.

\begin{definition}\label{SP}
For any $p\in S^{d_1}V_1\otimes S^{d_2}V_2$ let $\mathcal {S}({p})$ denote the set of all finite sets of points $S\subset \Y$
evincing $r({p})$, i.e. such that $[p]\in \langle \nu_{d_1,d_2} (S)\rangle$ and $\sharp (S) =r({p})$.
\end{definition}

\begin{definition}\label{DefAlphaBeta}
Let $[p_2]\in \mathbb{P}^{n_2}$ and $L$ be a line of $\mathbb{P}^{n_1}$, we call $L\times [p_2]\subset \mathbb{P}^{n_1}\times \mathbb{P}^{n_2}$ an \emph{$\alpha$-line}; while if we take $[p_1]\in \mathbb{P}^{n_1}$ and $L$ being a line of $\mathbb{P}^{n_2}$, we call $[p_1]\times L\subset \mathbb{P}^{n_1}\times \mathbb{P}^{n_2}$ a \emph{$\beta$-line}.  Moreover we will call $\PP^{n_1}\times [p_2]$ an \emph{$\alpha$-slice}, and $[p_1]\times \PP^{n_2}$ a \emph{$\beta$-slice}.
\end{definition}

\begin{notation}
In the sequel, the symbol ``~$\sqcup$ ~" indicates the disjoint union.
\end{notation}
 
\begin{theorem}[The decomposition theorem]\label{ee1.1}
Let $p \in S^{d_1}V_1 \otimes S^{d_2}V_2$ be a bi-homogeneous polynomial of rank $r(p)$ such that 
\begin{enumerate}[(a)]
\item\label{newa} either $2r({p}) \le 1+d_1+d_2$ and $|d_1-d_2| \le 2$,
\item\label{newb} or $r({p}) \le \min \{d_1,d_2\}$. 
\end{enumerate}
Assume that there exist two different sets of points $S, A$ evincing $r(p)$, i.e. $S,A\in \mathcal {S}({p})$. Then one of the following cases occurs:

\begin{enumerate}[(i)]

\item\label{primo} There are:

\begin{itemize} 
\item an integer $b$ with $2\le b \le (d_2+2)/2$ and $r({p}) \ge d_2+2 -b$,
\item a $\beta$-line $\mathcal{B}:= [p_1]\times L\subset \PP^{n_1}\times \mathbb{P}^{n_2}$, 
\item and a set of points $E\subset \Y$
\end{itemize} 
such that
$\sharp (E) = r({p}) +b-d_2-2$, $E\cap \mathcal{B} =\emptyset$,  $S = E\sqcup  \left(S\cap\mathcal{B}\right)$ and $A = E\sqcup \left(A\cap \mathcal{B}\right)$ (see Figure \ref{fig1});

\item\label{secondo} There are:

\begin{itemize}
\item an integer $b$ with $2\le b \le \lfloor (d_1+2)/2\rfloor $, $r({p}) \ge d_1+2-b$,
\item an $\alpha$-line $\mathcal{A}=R\times [p_2]\in \PP^{n_2}$, 
\item  and a set of points $F\subset \Y$
\end{itemize}
such that
$\sharp (F) = r({p}) +b-d_1-2$, $F\cap \mathcal{A}=\emptyset$,  $S = F\sqcup  \left(S\cap \mathcal{A}\right)$ and $A = F\sqcup \left(S\cap \mathcal{A}\right)$.

\item\label{third} We are in case (\ref{newa}) with $d_1\ne d_2$ and $p$ depends only on one factor. If $p$ depends only on the first factor, say $[p]$ is in the linear span of $\nu _{d_1,d_2}(\PP^{n_1}\times [p_2])$ with $[p_2]\in \PP^{n_2}$ and it is not as in the case (\ref{secondo}),
then $d_2>d_1$, $n_1\ge 2$, $r({p}) = d_1+1$, there is a plane $U\subseteq \PP^{n_1}$ and a reduced conic $C\subset U$, such that $[p]\in \nu _{d_1,d_2} (C\times \{p\})$ and all subsets of $\PP^{n_1}\times \PP^{n_2}$ evincing $r({p})$ are contained in $C$.

\end{enumerate}
In case (\ref{primo}) (resp. case (\ref{secondo})) there is a unique $Q\in \langle \nu_{d_1,d_2} (\mathcal{B})\rangle $ (resp. $Q\in \langle \nu_{d_1,d_2} (\mathcal{A})\rangle$) such that $r(Q) =d_2+2-b$
(resp.  $r(Q) =d_1+2-b$) and $[p]\in \langle \nu (E)\cup \{Q\}\rangle$ (resp. $[p]\in \langle \nu (F)\cup \{Q\}\rangle$) and for every
$M\in \Ss (Q)$ we have $E\cup M\in \Ss ({p})$ (resp. $F\cup M\in \Ss ({p})$) (see Figure \ref{fig2}).
\end{theorem}

\begin{figure}[!h]\label{fig1}
\centering
\includegraphics[width=0.5\textwidth]{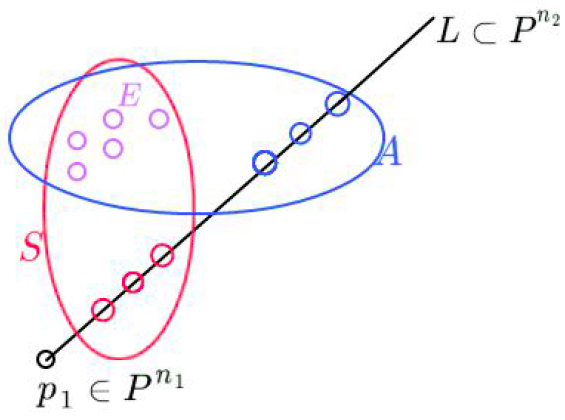} 
  \caption{\small{The schemes $S$ and $A$ computes the rank of $p$.  $S\cap A=E=\{\hbox{purple dots}\}$, $S=\{\hbox{red}+\hbox{purple dots}\}$, $A=\{\hbox{blue}+\hbox{purple dots}\}$.   In the figure we have dropped the ``~square brackets~" for $[p]$ to simplify the visualization.}}
\end{figure}

\begin{figure}[!h]\label{fig2}
\centering
\includegraphics[width=0.5\textwidth]{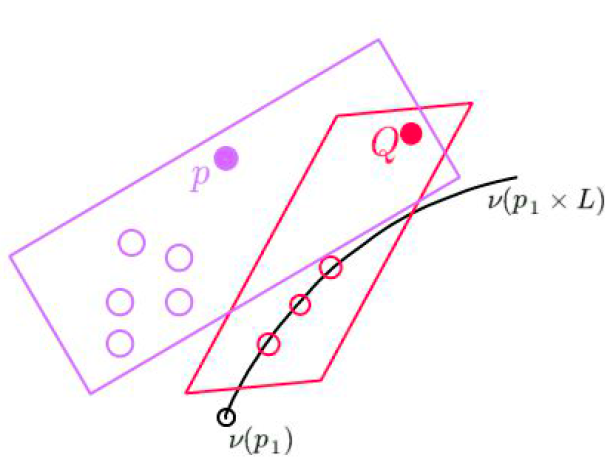} 
  \caption{\small{ In the figure we have dropped the ``~square brackets~" for $[p]$ and $[p_1]$ to simplify the visualization and $\nu$ stays for $\nu_{d_1, d_2}$. The point $[p]\in \langle \nu_{d_1,d_2}(E), Q\rangle$ where $ \nu_{d_1,d_2}(E)=\{\hbox{purple dots}\}$; $Q\in \langle \nu_{d_1,d_1}(S\setminus E)\rangle$ and $\nu_{d_1,d_1}(S\setminus E)=\{\hbox{red dots}\}$. There is a $g^1_{d_2+2-b}$ of points in $\nu_{d_1,d_2}([p_1]\times L)$ whose span contains $Q$ but such a $Q$ is unique as $E$ is in the decomposition of $p$.}}
\end{figure}

\begin{remark} By the classical result of Sylvester (see eg. \cite{cs,bgi, bcmt}) in cases (\ref{primo}) and (\ref{secondo}) the set $\mathcal {S}({p})$ is infinite.   Cases (\ref{primo}) and (\ref{secondo}) are mutually exclusive (see Remark \ref{usa3.01}).
\end{remark}

After having proved Theorem \ref{ee1.1}, we will show the uniqueness result that is described by the following proposition (case in which $p$ depends on both factors).

\begin{proposition}[Uniqueness of the decomposition]\label{ee7}
Take the assumptions of Theorem \ref{ee1.1} with $\sharp (\Ss ({p}))\ge 2$ and not in case (\ref{third}).

\begin{enumerate}[(a)]
\item To be in case (\ref{primo}) or in case (\ref{secondo}) and the value of the integer $b$ only depends on $p$, not on the
choice of $S,A\in \Ss ({p})$ with $S\ne A$.

\item
The $\beta$-curve 
$T =[p_1]\times L$ 
or  the $\alpha$-curve
$T = R\times [p_2]$
and the set $E:= S\setminus S\cap T =A\setminus A\cap T$ 
only depend on $[p]$, not the choice of $S, A\in \Ss ({p})$ with $S\ne A$.
\end{enumerate}
\end{proposition}

In the next subsection we collect all the proofs of Theorem \ref{ee1.1} and of Proposition \ref{ee7}.

\subsection{The proofs of Theorem \ref{ee1.1} and of Proposition \ref{ee7}}
Before giving the proof of Theorem \ref{ee1.1} we need some preliminary Lemma.

\begin{remark}\label{b5}
Fix  $(a_1,a_2)\in \NN ^2$, $T\in |\Oo _{\PP^1\times \PP^1} (1,0)|$ and a zero-dimensional scheme $Z\subset T$. Clearly $T\cong \PP^1$ and $\Oo _T(a_1,a_2)$ is a line bundle of degree
$a_2$. Hence
$h^1\left(T,\Ii _{Z,T}(a_1,a_2)\right) >0$ if and only if $\deg (Z)\ge a_2+2$. Since $h^i\left(\Oo _{\PP^1\times \PP^1}(a_1-1,a_2)\right) =0$ for $i=1,2$, we have $h^1\left(T,\Ii _{Z,T}(a_1,a_2)\right) >0$ if and only if
$h^1\left(\Ii _Z(a_1,a_2)\right) >0$.
\end{remark}

\begin{lemma}\label{b3}
Fix $F\in |\mathcal{O}_{\PP^1\times \PP^1}(1,1)|$ and integers $(a_1,a_2)\in \NN^2$. Let $Z\subset F$. We have $h^1(\Ii _Z(a_1,a_2)) >0$ if and only if either $\deg (Z)\ge a_1+a_2+2$ or
there is a proper subcurve $G$ of $F$, say of type $(e_1,e_2)$, with $\deg (Z\cap G) \ge e_2a_1+e_1a_2+2$.
\end{lemma}

\begin{proof}
The ``~if~'' part is true, because if $Z'\subseteq Z$, then $h^1(\Ii _{Z'}(a_1,a_2)) \le h^1(\Ii _Z(a_1,a_2))$ and $h^1(\Ii _Z(a_1,a_2)) =h^1(F,\Ii _{Z,F}(a_1,a_2))$.
\\
If $F$ is integral, then the lemma is obvious, because the arithmetic genus of $F$ is 0 
and $\deg (\Oo _F(a_1,a_2)) =a_1+a_2$. 
\\
Therefore assume $F=T+G$ with $T\in | \mathcal{O}_{\PP^1\times \PP^1}(1,0)|$ and $G\in |\mathcal{O}_{\PP^1\times \PP^1} (0,1)|$. If $\deg (Z\cap T)\le a_2+1$, then a residual sequence
$$0 \to \Ii _{\Res_T(Z)}(a_1-1,a_2) \to \Ii _{Z}(a_1,a_2) \to \Ii _{Z\cap T, T}(a_1,a_2) \to 0
$$
gives $h^1\left(\Ii _{\Res_T(Z)}(a_1-1,a_2)\right) >0$. Since $\Res _T(Z)\subset G$, Remark \ref{b5} gives $\deg (\Res _T(Z)) \ge a_1+1$ and hence $\deg (G\cap Z)\ge a_1+1$.
\\Similarly, by using the other exact sequence
$$0 \to \Ii _{\Res_G(Z)}(a_1,a_2-1) \to \Ii _{Z}(a_1,a_2) \to \Ii _{Z\cap G, G}(a_1,a_2) \to 0
$$
we get $\deg (\Res _G(Z)) \ge a_2+1$. 

Hence $\deg (Z)=\deg (\Res _G(Z))+\deg (Z\cap G)\ge a_1+a_2+2$.
\end{proof}

In Lemma \ref{ba1} we will need to perform and inductive procedure. The first step of the induction will be a consequence of the following lemma.

\begin{lemma}\label{1.7.1}
Fix $(a_1,a_2)\in \NN^2$. Let $Z\subset \PP^1\times \PP^1$ be a zero-dimensional scheme such that $\deg (Z)\le a_1+a_2+2$, $h^1\left(\Ii _{Z'}(a_1,a_2)\right) =0$
for every $Z'\subsetneq Z$ and $h^1\left(\Ii _Z(a_1,a_2)\right) >0$.
Then \begin{itemize}
\item
Either $\deg (Z)=a_2+2$ and there is $[o]\in \PP^1$ with $Z\subset \PP^1\times \{[o]\}$,
\item Or $\deg (Z)=a_1+2$ and there is $[q]\in \PP^1$
such that $Z\subset \{[q]\}\times \PP^1$,
\item Or $\deg (Z) =a_1+a_2+2$ and there is $F\in  |\mathcal{O}_{\PP^1\times \PP^1}(1,1)|$
such that $Z\subset F$,
\item Or $a_1 =0$ and $\deg (Z)= a_2+2$ or $a_2=0$ and $\deg (Z) =a_1+2$.
\end{itemize}
\end{lemma}

\begin{proof}
We use induction on $a_1+a_2$, the starting case of the induction being the trivial case $a_1=a_2=0$.

First assume $a_1=0$. If there is $[q]\in \PP^1$ with $\deg (Z\cap (\{[q]\} \times \PP^1)) \ge 2$, then
we are done, because $h^1\left(\Ii _{Z\cap (\{[q]\}\times \PP^1)}(0,a_2)\right) >0$.\\ Now assume $\deg (Z\cap (\{[q]\}\times \PP^1)) \le 1$ for all
$[o]\in \PP^1$. In this case the projection on the first factor $\pi _1:\mathbb{P}^1\times \mathbb{P}^1 \rightarrow \mathbb{P}^1$  induces an embedding of $Z$ into $\PP^1$ and we use that $h^1\left(\PP^1,\Ii _{\pi _1(Z),\PP^1}(a_2)\right) >0$ if and only if $\deg (\pi _1(Z))\ge a_2+2$.

Clearly the case $a_2=0$ is analogous.

Now assume $a_1>0$ and $a_2>0$. Fix $D\in |\Oo _{\PP^1\times \PP^1}(1,1)|$ such that $
\deg (D\cap Z)$ is maximal. If $Z\subset D$, then we apply Lemma \ref{b3} taking $F:= D$ if $\deg (Z)=a_1+a_2+2$.
Hence we may assume $Z\nsubseteq D$. 

Since
$Z\cap D\subsetneq Z$, we have $h^1\left(\Ii _{Z\cap D}(a_1,a_2)\right) =0$ and hence $h^1\left(D,\Ii _{Z\cap D}(a_1,a_2)\right)=0$. The residual exact sequence of $D$ in $\PP^1\times \PP^1$
$$0 \to \Ii _{\Res_D(Z)}(a_1-1,a_2-1) \to \Ii _{Z}(a_1,a_2) \to \Ii _{Z\cap D, D}(a_1,a_2) \to 0
$$
gives $h^1\left(\Ii _{\Res _D(Z)}(a_1-1,a_2-1)\right) >0$. We have $\deg ({\Res _D(Z)}) =\deg (Z)-\deg (D\cap Z)$.
Since $h^0\left(\Oo _{\PP^1\times \PP^1}(1,1)\right) =4$, we have 
$
\deg (D\cap Z)\ge 3$. Hence $\deg ({\Res _D(Z)}) \le (a_1-1)+(a_2-1)+1$. Let $W\subseteq \Res _D(Z)$ be a minimal subscheme such that $h^1\left(\Ii _W(a_1-1,a_2-1)\right) >0$.
Since $\deg (W)\le  (a_1-1)+(a_2-1)+1$, the inductive assumption gives that
\begin{enumerate}[(a)]
\item Either $a_1=1$ and $\deg (W) \ge (a_2-1)+2$,
\item Or $a_2=1$ and $\deg (W)\ge (a_1-1)+2$, 
\item Or $a_1\ge 2$ and there is $[o]\in \PP^1$ with $\deg (W\cap \PP^1\times \{[o]\}) \ge a_2+1$, 
\item Or $a_2\ge 2$ and there is $[q]\in \PP^1$
such that $\deg (W\cap \{[q]\}\times \PP^1)\ge a_1+1$.
\end{enumerate}

Note that in each case the inequality holds if we take $\Res _D(Z)$ instead of $W$.

First assume $a_1=1$ and $\deg ({\Res _D(Z)}) \ge a_2+1$. Since $\deg (Z) \le a_1+a_2+2 =a_2+3$, we get $\deg (Z\cap D)\le 2$, a contradiction.

In the same way we conclude if $a_2=1$ and $\deg ({\Res _D(Z)}) \ge a_1+1$.

Now assume $a_1\ge 2$ and the existence of $[o]\in \PP^1$ with $\deg \left({\Res _D(Z)}\cap \PP^1\times \{[o]\}\right) \ge a_2+1$. 
$\hbox{Set }R:= \PP^1\times \{[o]\}\in |\Oo _{\PP^1\times \PP^1}(0,1)|$.
If $\deg (R\cap Z)\ge a_2+2$, then we are
done. Hence we may assume $\deg (R\cap Z)\le a_2+1$. Since $Z\supset {\Res _D(Z)}$ and $\deg ({\Res _D(Z)}\cap R) \ge a_2+1$, we get $Z\cap R =Z\cap {\Res _D(Z)}$ and
$\deg (Z\cap R) =a_2+1$. 

We have $\deg ({\Res _R(Z)}) \le a_1+1$. The residual exact sequence of $R$
$$0 \to \Ii _{\Res_R(Z)}(a_1,a_2-1) \to \Ii _{Z}(a_1,a_2) \to \Ii _{Z\cap R, R}(a_1,a_2) \to 0
$$
 gives
$h^1\left(\Ii _{\Res _R(Z)}(a_1,a_2-1)\right) >0$. 

Let $W'\subseteq \Res _R(Z)$ be a minimal subscheme with $h^1\left(\Ii _{W'}(a_1,a_2-1)\right) >0$. The inductive assumption gives $\deg ({\Res _R(Z)}) \ge \deg (W')
\ge 2+\min \{a_1,a_2-1\}$. Since $\deg ({\Res _R(Z)}) \le a_1+1$, we get $a_2\le a_1-2$. Take $L\in |\Oo _{\PP^1\times \PP^1}(1,0)|$ such that $\deg (L\cap Z )$ is maximal. Since
$h^1\left(\Ii _{Z'}(a_1,a_2)\right) =0$ for all $Z'\subsetneq Z$, either $Z\subset L$ (and hence $\deg (Z) =a_2+2$ and the lemma is true) or $\deg (L\cap Z) \le a_2+1$. We
may assume that $\deg (L\cap Z) \le a_2+1$ and hence
$h^1\left(L,\Ii _{L\cap Z}(a_1,a_2)\right)=0$. The residual exact sequence of $L$ gives $h^1\left(\Ii _{\Res_L(Z)}(a_1-1,a_2)\right) >0$.
Let $W_1\subseteq \Res _L(Z)$ be a minimal subscheme such that $h^1\left(\Ii _{W_1}(a_1-1,a_2)\right) >0$. Since $a_1-1> 0$, the inductive assumption gives
that either $\deg (W_1) =a_1+1$ and $W_1$ is contained in $R_1\in |\Oo _{\PP^1\times \PP^1}(1,0)|$ or $\deg (W_1) =a_2+2$ and $W_1$ is contained in an element of
$|\Oo _{\PP^1\times \PP^1}(0,1)|$ or $\deg (W_1) = a_1+a_2+1$ and $W_1$ is contained in an element of $|\Oo _{\PP^1\times \PP^1}(1,1)|$. In the latter case
we get $\deg (D\cap Z) \ge a_1+a_2+1$ and so $\deg (\Res _D(Z)) \le 1$ and hence $h^1(\Ii _{\Res _D(Z)}(a_1-1,a_2-1)) =0$, a contradiction. In the second
case we get that we are in the first case of the lemma. Now assume the existence of $R_1\in |\Oo _{\PP^1\times \PP^1}(1,0)|$ such that $\deg (R_1\cap W_1) \ge a_1+1$.
Since $W_1\subseteq \Res _L(Z)\subset Z$, the maximality property of the integer $\deg (L\cap Z)$ gives $\deg (L\cap Z) \ge a_1+1$. Therefore $\deg (Z) \ge \deg (L\cap Z)+\deg (\Res _L(Z))
\ge 2a_1+2$, contradicting the inequalities $a_2\le a_1-2$ and $\deg (Z)\le a_1+a_2+2$.

The same proof works if $a_2\ge 2$ and there is $[q]\in \PP^1$ such that $\deg ({\Res _F(Z)}\cap \{Q\}\times \PP^1)\ge a_1+1$.
\end{proof}

\begin{lemma}\label{ba1}
Let  $\Gamma\subset \PP^{n_1}\times \PP^{n_2}$ be  zero-dimensional scheme such that $\deg (\Gamma)\le d_1+d_2+1$, $h^1(\Ii _{\Gamma '}(d_1,d_2)) =0$ for
all $\Gamma '\subsetneq \Gamma$ and $h^1(\Ii _\Gamma(d_1,d_2)) >0$ with $ d_1,d_2>0$. Then either there is $[p_1]\in \PP^{n_1}$ such that $\Gamma \subset [p_1]\times \PP^{n_2}$ or there is $[p_2]\in \PP^{n_2}$ such that $h^1(\Ii _{\PP^{n_1}\times [p_2]}(d_1,d_2)) >0$. If the second case occurs and $d_2\ge d_1$, then $\deg (\Gamma )=d_2+2$ and there
is a $\beta$-line $T$ such that $\Gamma \subset T$.
\end{lemma}

\begin{proof}
The last sentence follows from the first part of the lemma by \cite[Lemma 34]{bgi}, because $\deg (\Gamma)\le 2d_2+1$ if $d_1\le d_2$. Hence it is sufficient to prove the first part. By assumption $h^1(\Ii _{\Gamma '}(d_1,d_2)) =0$ for all $\Gamma'\subsetneq \Gamma$. With this assumption we need to prove that $\Gamma$ is contained in one of the slices of $\PP^{n_1}\times \PP^{n_2}$. By Lemma \ref{b3} and Lemma \ref{1.7.1} we may assume $n_1+n_2>2$ and use induction on the integer $n_1+n_2$. We also use induction on the integer $d_1+d_2$, the case $(d_1,d_2)=(1,1)$ being obviously true, because $\deg (\Gamma )\le 3$
(but note that as stated the result would be wrong if $d_2= 0$ and $n_1\ge 2$). With no loss of generality for the firs part we may assume $d_2\ge d_1$ and in particular $d_2\ge 2$. 

Take $D_1\in |\Oo _{\Y}(0,1)|$ such that $f_1:= \deg (\Gamma\cap D_1)$ is maximal. Since obviously $h^0(\Oo _{\Y}(0,1)) = n_2+1$, we have $f_1\ge n_2>0$. If
$h^1\left(D_1,\mathcal {I}_{\Gamma\cap D_1}(d_1,d_2)\right) >0$, then we may use the inductive assumption on the integer $n_1+n_2$. Hence we  assume that
$h^1\left(D_1,\mathcal {I}_{\Gamma\cap D_1}(d_1,d_2)\right)=0$. Therefore by the 
Castelnuovo's sequence 
$$0 \to \Ii _{\Res_{D_1}(\Gamma)}(d_1,d_2-1) \to \Ii _{\Gamma}(d_1,d_2) \to \Ii _{\Gamma\cap {D_1}, {D_1}}(d_1,d_2) \to 0
$$
we have $h^1\left(\Ii _{\mathrm{Res}_{D_1}(\Gamma)}(d_1,d_2-1)\right) =0$.
\\Now,
let $\pi_i: \mathbb{P}^{n_1}\times \mathbb{P}^{n_2}\to \mathbb{P}^{n_i}$ be the projection on the $i$-th factor for $i=1,2$. 
Since $f_1>0$ and $d_2-1 >0$, the inductive assumption gives that either there is a point $[p_1]\in \PP^{n_1}$ such that $h^1\left(\Ii _{\mathrm{Res}_{D_1}(\Gamma)\cap \pi _1^{-1}([p_1])}(d_1,d_2-1)\right) >0$ or there is a point $[p_2]\in \PP^{n_2}$ such that $h^1\left(\Ii _{\mathrm{Res}_{D_1}(\Gamma)\cap \pi _2^{-1}([p_2])}(d_1,d_2-1)\right) >0$. If $[p_2]$ exists, then we are done, since $\Gamma\supseteq \mathrm{Res}_{D_1}(\Gamma)$.

Now assume that such a $[p_2]$ does not exist while suppose the existence of $[p_1]\in \mathbb{P}^{n_1}$ such that $h^1\left(\Ii _{\mathrm{Res}_{D_1}(\Gamma)\cap \pi _1^{-1}([p_1])}(d_1,d_2-1)\right) >0$. Since $f_1>0$ and $d_2\ge d_1$ we have $\deg (\mathrm{Res}_{D_1}(\Gamma)) \le 2d_2$.
By \cite[Lemma 34]{bgi}  there is a $\beta$-line $T\subset \pi _1^{-1}([p_1])$ such that $\deg (T\cap \mathrm{Res}_{D_1}(\Gamma)) \ge d_2+1$.
\\ If $n_2\ge 2$, there is $D\in |\Oo _{\Y}(0,1)|$ containing
$T$ and hence $f_1\ge d_2+1$. We get that  $\deg (\Gamma)\ge 2d_2+2$ which contradicts the hypothesis. 
\\ Now assume $n_2=1$ and hence $n_1\ge 2$. Fix $M_1\in |\Oo _{\Y}(1,0)|$ such that
$g:= \deg (M_1\cap \Gamma)$ is maximal. The existence of the $\beta$-line $T$ such that $\deg (T\cap \mathrm{Res}_{D_1}(\Gamma)) \ge d_2+1$ gives that $g\ge d_2+n_1-1$. 
\\If $h^1\left(M_1,\mathcal {I}_{\Gamma\cap M_1}(d_1,d_2)\right) >0$, then, again, we can use the inductive assumption on the integer $n_1+n_2$. 
\\ Hence we may assume that
$h^1\left(M_1,\mathcal {I}_{\Gamma\cap M_1}(d_1,d_2)\right)=0$. The Castelnuovo's sequence gives $h^1\left(\mathcal {I}_{\mathrm{Res}_{M_1}(\Gamma)}(d_1-1,d_2)\right) >0$. 
\\
If $d_1=1$, then
we get $\deg (\mathrm{Res}_{M_1}(\Gamma))\ge 2$ and hence $\deg (\Gamma)\ge g+2 \ge d_2+n_1+1>d_1+d_2+1$, which contradicts the hypothesis on the degree of $\Gamma$. 
\\ If $d_1\ge 2$, the inductive assumption
gives $\deg (\Res _{M_1}(\Gamma)) \ge d_1+2$ and than we have $\deg (\Gamma) \ge d_2+n_1-1 +d_1+2\ge d_1+d_2+2$, which is again a contradiction.
\end{proof}

The case $n=2$ of the following observation is \cite[Lemma 4.4]{bb3}; the case $n>2$ follows by induction on $n$ taking a hyperplane $H\subset \PP^n$ such that $\deg (Z\cap H)$ is maximal.

\begin{remark}\label{lll1}
Let $Z\subset \PP^n$ be a finite set such that $h^1(\Ii _Z(t)) >0$ and $\deg (Z)\le 2t+2$. Then either there is a line $L\subset \PP^n$ with $\deg (L\cap Z)\ge d+2$
or $\deg (Z) =2t+2$ and there is a reduced conic $C\subset \PP^n$ such that $Z\subset C$.
\end{remark}

\begin{remark}\label{usa3.01}
Take $\Gamma,d_1,d_2,n_1,n_2$  as in Lemma \ref{ba1} and assume the existence of a $\beta$-line $\mathcal{B}$ 
such that $h^1(\Ii _{\Gamma\cap \mathcal{B}
}(d_1,d_2)) >0$,
i.e. such that $\deg (\Gamma\cap \mathcal{B}
) \ge d_1+2$. Fix any $\alpha$-line $\mathcal{A}$
. Since $\deg (\mathcal{B}
\cap \mathcal{A}
) \le 1$,
we have $\deg (\Gamma\cap \mathcal{A}
) \le d_1+d_2+1 -(d_1+2)+1 = d_2$. 
\end{remark}

We are now ready to prove the decomposition Theorem \ref{ee1.1}.

\begin{proof}[Proof of Theorem \ref{ee1.1}:]
Since the proof of this theorem is quite structured, we decided to divide  it in various claims in order to facilitate the reading and to equip each one of them with a figure.

\medskip

First of all remark that we have $1+2\min \{d_1,d_2\}\le 1+d_1+d_2$ and hence with any of the assumptions of Theorem \ref{ee1.1} we could get  $2r({p}) \le 1+d_1+d_2$.

\medskip

Let's start by fixing two different sets of points $S, A\in \Ss ({p})$ computing the rank of $p$.
Then let 
$$S'':= S\cap A.$$ 
as in Figure 3.

\begin{figure}[!h]\label{fig3}
\centering
\includegraphics[width=0.4\textwidth]{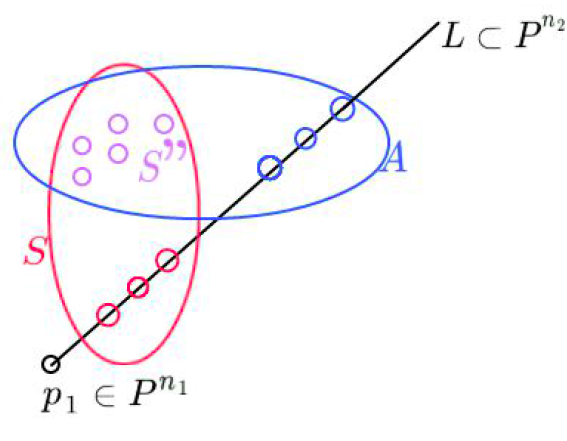} 
  \caption{\small{In the figure we have dropped the ``~square brackets~" for $[p]$ to simplify the visualization.}}
\end{figure}

Since $S$ and $ A$  are different, then $S''$ is a proper subset of both $S$ and $A$, i.e.  $S''\subsetneq S,A$. 

\medskip

\quad {\emph {Claim 1: }} Take any subset of points $G\subseteq S''$. There is a unique point $Q\in \langle \nu_{d_1,d_2} (A\setminus G)\rangle \cap \langle \nu_{d_1,d_2} (S\setminus G)\rangle$ such that $[p]\in \langle
\nu_{d_1,d_2} (G) \cup\{Q\}\rangle$ and $r(Q) = \sharp (S) -\sharp (G)$ (this is illustrated in Figure 4).

\begin{figure}[!h]\label{fig4}
\centering
\includegraphics[width=0.4\textwidth]{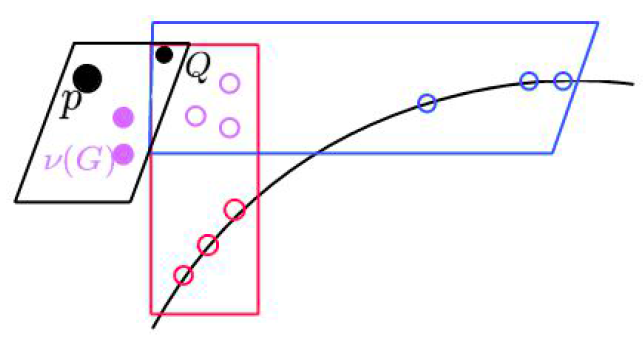} 
  \caption{\small{In the figure we have dropped the ``~square brackets~" for $[p]$ to simplify the visualization and $\nu$ stays for $\nu_{d_1, d_2}$.}}
\end{figure}

\quad {\emph {Proof of Claim 1:}} If $G=\emptyset$, then this claim is trivial (it is sufficient to take $Q=[p]$). So we may assume $G\ne \emptyset$. 

Since $\nu_{d_1,d_2} (S)$ is linearly independent, we have $\langle \nu_{d_1,d_2} (S)\rangle \cap \langle \nu_{d_1,d_2} (A)\rangle = \langle \nu_{d_1,d_2} (S\setminus G)\rangle \cap \langle \nu_{d_1,d_2} (A\setminus G)\rangle+ \langle \nu_{d_1,d_2} (G)
\rangle$ and this
is a direct decomposition.
Since $[p]\notin \langle \nu_{d_1,d_2} (G)\rangle$ for any $G\subsetneq S$, we have $[p]\notin \langle \nu_{d_1,d_2} (S''
)\rangle$ and so there is a unique $Q\in \langle \nu_{d_1,d_2} (A\setminus G)\rangle \cap \langle \nu_{d_1,d_2} (S\setminus G)\rangle$ such that $[p]\in \langle
\nu_{d_1,d_2} (G) \cup\{Q\}\rangle$. Since $Q\in \langle \nu_{d_1,d_2} (S\setminus G)\rangle$, we have $r(Q) \le \sharp (S) -\sharp (
G)$.
Since $[p]$ is in the linear span of $Q$ and $\nu_{d_1,d_2} (G
)$, we have $r({p}) \le r(Q)+\sharp (G)$. Hence $r(Q) = \sharp (S) -\sharp (G)$.\qed

\medskip

Now, set 
\begin{equation}\label{B}B:= A\cup S.\end{equation}
Since $2r({p}) \le 1+d_1+d_2$, we have $\sharp (B)\le 1+d_1+d_2$.
Since $A\nsubseteq S$ and $S\nsubseteq A$, we have
$h^1(\Ii _B(d_1,d_2)) >0$ (\cite[Lemma 1]{bb}). By Lemma \ref{ba1} there is either a $\beta$-slice $\mathbb{B}$
such that $h^1(\Ii _{B\cap\mathbb{B}
}(d_1,d_2))>0$
or an $\alpha$-slice $\mathbb{A}$ 
such that  $h^1(\Ii _{B\cap \mathbb{A}
}(d_1,d_2))>0$. We assume the existence of the $\beta$-slice $\mathbb{B}=[p_1]\times \PP^{n_2}$, because the case of the $\alpha$-slice is analogous. We first assume
that $\deg (\Gamma \cap \mathbb{B})\ge 2d_2+2$. We get that we are in case (\ref{newa}) with $d_1\ge d_2$ and that $B\subset \mathbb {B}$. By Remark \ref{lll1} we have $n_2\ge 2$ and there is a reduced conic
$C\subset \PP^{n_2}$ such that $[p]\in \nu _{d_1,d_2}([p_1]\times C)$ and $A\cup S\subset [p_1]\times C$, i.e. we are in case (\ref{third}). Now assume $\deg (\Gamma \cap \mathbb{B})\le 2d_2+1$. By \cite[Lemma 34]{bgi} there is a line $L\subset \PP^{n_2}$ such that $\deg (\Gamma \cap [p_1]\cap L )\ge d_2+2$. Set $\mathcal {B}:= [p_1]\times L$.

Set:
\begin{itemize}
\item $A':= A\cap \mathbb {B}$, 
\item $S':= S\cap \mathbb {B}$  and 
\item $B':= B\cap \mathbb {B} =A'\cup S'$.
\end{itemize}

Since $h^1(\mathbb {B},\Ii _{B'}(d_1,d_2)) >0$, we have $\sharp (B')\ge d_2+2$ and equality holds only if $B'$ is contained in a line.

\medskip

\quad {\emph {Claim 2: }} We have that $A\setminus A'=S\setminus S'$ (illustrated in Figure 5).

\begin{figure}[!h]\label{fig5}
\centering
\includegraphics[width=0.4\textwidth]{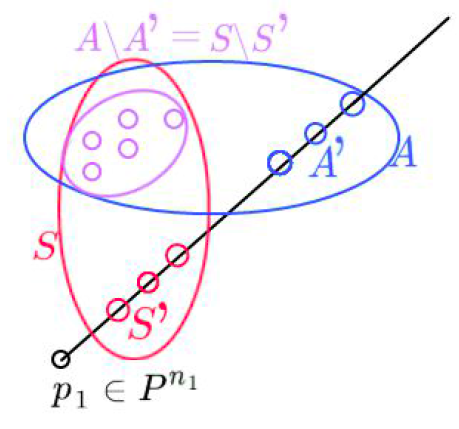} 
  \caption{\small{In the figure we have dropped the ``~square brackets~" for $[p]$ to simplify the visualization.}}
\end{figure}

\quad {\emph {Proof of Claim 2:}} Let $D\subset \PP^{n_1}$ be a general hyperplane containing $[p_1]$. For a general $D$ we have $D\cap B = B'$.

Consider the residual exact sequence of $D$:
\begin{equation}\label{eqee1}
0 \to \Ii _{B\setminus B'}(d_1-1,d_2) \to \Ii _B(d_1,d_2) \to \Ii _{B',D}(d_1,d_2)\to 0.
\end{equation}

If $h^1\left(\Ii _{B\setminus B'}(d_1-1,d_2)\right) =0$, then, from \cite[Lemma 5.1]{bb2}, we immediately get that $A\setminus A' = S\setminus S'$. 

Now assume $h^1\left(\Ii _{B\setminus B'}(d_1-1,d_2)\right)>0$. Since $\sharp (B) \le d_1+d_2+1$ and $\sharp (B') \ge d_2+2$, we have $\sharp (B\setminus B') \le d_1-1$.
Hence $h^1\left(\Ii _{B\setminus B'}(d_1-1,d_2)\right) =0$ if $d_1=1$. Therefore we may assume $d_1-1 >0$. By Lemma \ref{ba1} (with $\Gamma$ being $B$) we have $d_2\le d_1-3$ and $\sharp (B\setminus B') \ge d_2+2$ (one can also apply Remark \ref{usa3.01} with $\Gamma=B$ and $\mathcal{B}=N$), contradicting the assumption 
$\sharp(B)\geq 2d_2$
when $d_2\le d_1-3$.\qed

\medskip

Now if we apply Claim 1 to the set $G:= A\setminus A'$, we get a unique $Q\in \langle \nu_{d_1,d_2} ([p_1]\times \PP^{n_2})\rangle$ with $A'\in \Ss (Q)$ and
$[p]\in \langle \nu_{d_1,d_2} (G)\cup \{Q\}\rangle$. 
\\
Let $L\subseteq \PP^{n_2}$
be a line such  $h^1\left(\Ii _{B\cap ([p_1]\times L)}(d_1,d_2)\right)>0$.
With $\mathcal{B}=[p_1]\times L$, set:

\begin{itemize}
\item $A_1:= A\cap  \mathcal{B}$, 
\item $S_1:= S\cap \mathcal{B}$ and 
\item $B_1:= A_1+S_1$.
\end{itemize}

\medskip

\quad {\emph {Claim 3: }} We have $A\setminus A_1 = S\setminus S_1$ (illustrated in Figure 6).

\medskip

\begin{figure}[!h]\label{fig6}
\centering
\includegraphics[width=0.4\textwidth]{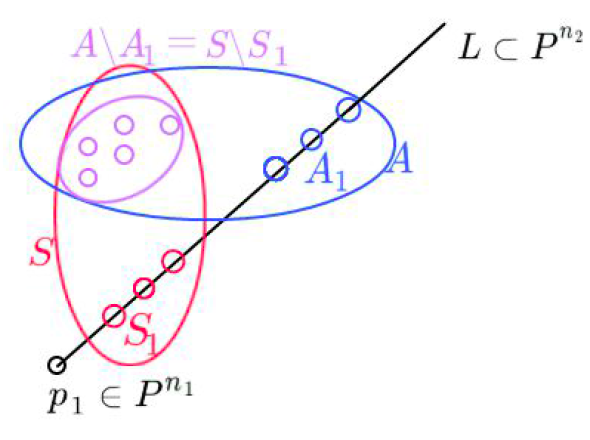} 
  \caption{\small{In the figure we have dropped the ``~square brackets~" for $[p]$ to simplify the visualization.}}
\end{figure}

\quad {\emph {Proof of Claim 3:}} If $n_2 =1$, then $L=\PP^{n_2}$, $A_1=A'$ and $S_1=S'$ and we may apply Claim 2. 

Now assume $n_2\ge 2$ and fix $H\in |\Oo _{\Y}(0,1)|$ with $L\subset H$ and $H$ general. For a general
$H$ we have $S\cap H=S_1$ and $A\cap H=A_1$. Consider the residual exact sequence of $H$:
\begin{equation}\label{eqee2}
0 \to \Ii _{B\setminus B_1}(d_1,d_2-1) \to \Ii _B(d_1,d_2)\to \Ii _{B_1,H}(d_1,d_2)\to 0.
\end{equation}
Since $\sharp (B_1) \ge d_2+2$, we have $\sharp (B\setminus B_1)  \le d_1-1$. 

If $h^1(\Ii _{B\setminus B_1}(d_1,d_2-1)) =0$, then 
\cite[Lemma 5.1]{bb2} gives $A\setminus A_1=S\setminus A_1$.

Now assume $h^1\left(\Ii _{B\setminus B_1}(d_1,d_2-1)\right)>0$. First assume $d_2\ge 2$. Lemma \ref{1.7.1} for the integers $(a_1,a_2) =(d_1,d_2-1)$ gives $\sharp (B\setminus B_1) \ge 2 +\min \{d_1,d_2-1\}$. Since
$\sharp (B\setminus B_1) \le d_1-1$, we first get $d_2 \le d_1$ and then (since $|d_1-d_2|\le 2$, $\sharp (B)\le 1+d_1+d_2$ and $\sharp (B_1)\ge d_2+2$) we get $\sharp (B) = 1+d_1+d_2$, $\sharp (B_1)=d_2+2$ and $d_1=d_2+2$. Since we are necessary in case (\ref{newa}) $\sharp (B)\le 2r({p}) \le 1+d_1+d_2$, we get $2r({p}) =1+d_1+d_2= 3+2d_2$, a contradiction (because $3+2d_2$ is odd). Now assume
$d_2=1$ and hence $d_1\le 3$. Since $2r({p}) \le 1+d_1+d_2\le 5$, we have $r({p}) \le 2$. Since $B\supseteq B_1$ and $\sharp (B_1)\ge d_2+2$, we get a contradiction. \qed

\medskip

Now set 
$$E:= S\setminus S_1.$$ 
as in Figure 7.

\begin{figure}[!h]\label{fig7}
\centering
\includegraphics[width=0.4\textwidth]{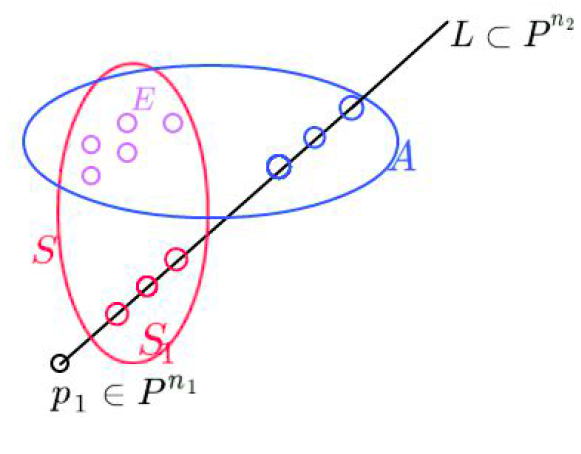} 
  \caption{\small{In the figure we have dropped the ``~square brackets~" for $[p]$ to simplify the visualization.}}
\end{figure}

By Claim 1 applied to the set $G:= E$, we have $A_1,S_1\in \Ss (Q)$ and $A_1\ne S_1$. By the famous theorem of Sylvester \cite{cs, bgi} either $d_2$ is
even and $\sharp (A_1)=\sharp (S_1) = (d_2+2)/2$ or the border rank $b$ of $Q$ is smaller that the rank of $Q$, $2 \le b \le \lfloor d_2/2\rfloor$
and $r(Q) =d_2+2-b$. In the latter case we are in case (\ref{primo}) of Theorem \ref{ee1.1} with $b$ the border rank. If $\sharp (S_1) = (d_2+2)/2$ we are in case (\ref{primo}) with
$b= (d_2+2)/2$. Both cases are contained in case (\ref{primo}) of Theorem \ref{ee1.1}.\end{proof}

\medskip

With the proof of Theorem \ref{ee1.1} done, we can now show the uniqueness part and prove Proposition \ref{ee7}.

\medskip

\begin{proof}[Proof of Proposition \ref{ee7}:]
Fix $S, A, S', A'\in \Ss ({p})$ with $S\ne A$, $S'\ne A'$ and $\{S,A\} \ne \{S',A'\}$. With no loss of generality we may assume that $(S,A)$ is associated to the $\beta$-line $
T= [p_1]\times L$, i.e. we are in case (\ref{primo}) of Theorem \ref{ee1.1}.

\begin{enumerate}
\item\label{(i)} First assume $\{S,A\} \cap  \{S',A'\} \ne \emptyset$, say $A=A'$ ($S=S'$ will be analogous).

In the contest of Theorem \ref{ee1.1},  assume that $(S',A)$ is associated to a $\beta$-line $T'$,
an integer $b'$ and a set $E' = S'\setminus S'\cap T' = A\setminus A\cap T'$. 

Clearly, since $A=A'$ and since they are both contained in a line ($T$ and $T'$ respectively) the two lines have to be the same: $T=T'$. This implies that $E =A\setminus A\cap T =E'$, therefore in this case there is nothing to prove.

\item Now assume $\{S,A\} \cap  \{S',A'\} =\emptyset$. We can apply step (\ref{(i)}) to the two pairs $(S,A)$ and $(S,A')$ (resp. the two pairs $(S,A)$ and $(S',A)$)
and get that  $A'\setminus A'\cap T = E$ (resp. $S'\setminus S'\cap T =E$). Therefore $S'\cap A'\supseteq E =S\cap A$. By symmetry, we also have that $S'\cap A' =S\cap A$.
If $(S',A')$ is associated to a curve $T'$ and the integer $b'$, then  $b=b'$.
\end{enumerate}
\end{proof}

\subsection{A trivial bound for $\rho'(X)$ in the case of Segre-Veronese varieties}
In the Introduction, in Definition \ref{rho}, we introduced $\rho ' (X)$ to be the maximal integer $t$ such that any subset of $X$ with cardinality $t$ is linearly independent. Then we stated as a general fact that if $X$ is the Segre-Veronese variety of $k$ factors, i.e. $X=\SV$  is the embedding of $\mathbb{P}(V_1)\times \cdots \times \mathbb{P}(V_k)$ into $\mathbb{P}(S^{d_1}V_1\otimes \cdots \otimes  S^{d_k}V_k)$ induced by the complete linear system $|\mathcal{O}_{\mathbb{P}(V_1)\times  \cdots \times \mathbb{P}(V_k)}(d_1,  \ldots , d_k)|$,
then
$\rho '(X) = 1 +\min _{1\le i \le k} \{d_i\}$ (Lemma \ref{suck}). Unfortunately we cannot find a precise reference for this fact, but since it is quite easy to be shown, we include the proof for sake of completeness.

\begin{lemma}\label{suck}
Let $X$ be the Segre-Veronese embedding of $ \mathbb {P}^{n_1}\times \cdots \times \mathbb {P}^{n_k}$ into  $\mathbb{P}(S^{d_1}V_1\otimes \cdots \otimes  S^{d_k}V_k)$ by
the linear system $|\Oo _{ \mathbb {P}^{n_1}\times \cdots \times \mathbb {P}^{n_k}}(d_1,\dots ,d_k)|$. Then $\rho '(X) =1+\min _{1\le i \le k}\{d_i\}$.
\end{lemma}

\begin{proof}
With no loss of generality we may assume $d_i\ge d_1$ for all $i$. 
\\
Fix a line $L\subseteq \mathbb {P}^{n_1}$ and $O_i\in \mathbb {P}^{n_i}$, $i=2,\dots ,k$. Take $E\subset L$
with $\sharp (E)=d_1+2$ and set $F:= E\times \{O_2\}\times \cdots \times \{O_k\}$. 
\\
Since $h^1({ \mathbb {P}^{n_1}\times \cdots \times \mathbb {P}^{n_k}},\Ii _F(d_1,\dots ,d_k)) = h^1(\mathbb {P}^{n_1},\Ii _E(d_1)) = h^1(L,\Ii _E(d_1)) =1$,
we have $\rho '(X) \le 1+\min _{1\le i \le k}\{d_i\}$. 

To prove the lemma it is sufficient to show that $$h^1({ \mathbb {P}^{n_1}\times \cdots \times \mathbb {P}^{n_k}},\Ii _S(d_1,\dots ,d_k)) =0$$ for every set
$S\subset { \mathbb {P}^{n_1}\times \cdots \times \mathbb {P}^{n_k}}$ with $\sharp (S) \le d_1+1$.
\\ Order the points $p_1,\dots ,p_x$, $x\le d_1+1$, of $S$. Set $S_0:= \emptyset$ and $S_y = \{p_1,\dots ,p_y\}$. Since $\Oo _{ \mathbb {P}^{n_1}\times \cdots \times \mathbb {P}^{n_k}}(1,\dots ,1)$ is very ample, for each
$i=1,\dots ,x-1$ there is $H_i\in |\Oo _{ \mathbb {P}^{n_1}\times \cdots \times \mathbb {P}^{n_k}}(1,\dots ,1)|$ with $p_i\in H_i$ and $p_j\notin H_i$ for all $j\ne i$. If $d_i=1$ for all $i$, then set $M:= \emptyset$.
If $d_i\ne 1$ for some $i$, let $M\in |\Oo _{ \mathbb {P}^{n_1}\times \cdots \times \mathbb {P}^{n_k}}(d_1-1,\dots ,d_k-1)|$ with $M\cap S=\emptyset$.
The divisors $H_1+M,\dots ,H_{x-1}+M$ give $h^0({ \mathbb {P}^{n_1}\times \cdots \times \mathbb {P}^{n_k}},\Ii _{S_y}(d_1,\dots ,d_k)) = h^0({ \mathbb {P}^{n_1}\times \cdots \times \mathbb {P}^{n_k}},\Ii _{S_{y-1}}(d_1,\dots ,d_k)) -1$ for $i=1,\dots ,x$.
Hence $h^1(\Ii _S(d_1,\dots ,d_k))=0$.
\end{proof}


\section{Rank on the tangential variety of Segre-Veronese varieties}\label{RankTg}

First of all in this section we will consider the Segre-Veronese variety\\ $\SV$ of any number of factors. Then we will describe the rank of multi-homogeneous polynomials (partially symmetric tensors) that can be written as a limit of a sequence of rank 2  multi-homogeneous polynomials (partially symmetric tensors). 
If $p\in S^{d_1}V_1^*\otimes \cdots \otimes S^{d_k}V_k^*$ is one of those polynomials, one says  that $[p]$ has \emph{border rank} 2.  To be more precise, let 
$$\sigma_2(\SV)=\overline{\bigcup_{[p_1],[p_2]\in \SV}\langle [p_1],[p_2]\rangle}$$
be the \emph{secant variety} to $\SV$.\\ Clearly $\SV\subset \sigma_2(\SV)$. \\An element in $\sigma_2(\SV)$ that is not in $\SV$ is either a projective class of a  multi-homogeneous polynomial (partially symmetric tensor) of rank 2, or it is the limit of a sequence of rank 2 elements. Clearly, from the point of view of the knowledge of the rank, the only interesting case is the one of points that are limit of rank 2 elements. Those represent a closed subvariety of $\sigma_2(\SV)$ that we indicate with $\tau(\SV)$ and that is the \emph{tangential variety} of $\SV$:
$$\tau(\SV)=\overline{\bigcup_{[p]\in \SV }T_{[p]}(\SV)}.$$
Here we prove the following theorem.
 
\begin{theorem}\label{i1}
The rank of $[p]\in \tau(\SV)$ is 
$$r_{d_1, \ldots , d_k}(p)=\sum_{i=1}^h d_i$$ 
if $V_1^*=K[x_{1,0}, \ldots , x_{1,n_1}]_1$, $\ldots ,$ $V_h^*=K[x_{h,0}, \ldots , x_{h,n_h}]_1$ are the minimum sets of variables to which the multi-homogeneous polynomial $[p]$ actually depends on, $h\leq k$.  
\\
In terms of partially symmetric tensors this means that the tensor depends actually on $h\leq k$ factors: $p\in S^{d_1}V_1\otimes \cdots \otimes S^{d_h}V_h\subset S^{d_1}V_1\otimes \cdots \otimes S^{d_k}V_k$.
\end{theorem}

This result is expected, in fact it is the generalization of the following two particular and well known cases. 

If $k=1$ then  $\SV=S_{d_1}(V_1)$ is nothing else than the  Veronese variety obtained by embedding $\mathbb{P}(V_1)$ with the complete linear system $|\mathcal{O}_{\mathbb{P}(V_1)}(d_1)|$ into $\mathbb{P}(S^{d_1}V_1)$ that parameterizes projective classes of rank 1  homogeneous polynomials of degree $d_1$ in $n_1+1$ variables that are pure powers of linear forms (completely symmetric tensors of order $d_1$). In this case the rank of $[p]\in\tau(S_{d_1}(V_1))\setminus S_{d_1}(V_1)$ is equal to $d_1$ (this is done in \cite{bgi}). 

The other particular case is the one where $d_1=\cdots =d_k=1$. It corresponds to Segre variety where $\SV=S_{1, \ldots , 1}(V_1, \ldots , V_k)$ is the embedding of $\mathbb{P}(V_1)\times \cdots \times \mathbb{P}(V_k)$ with the complete linear system $|\mathcal{O}_{\mathbb{P}(V_1)\times \cdots \times \mathbb{P}(V_k)}(1, \ldots, 1)|$ into $\mathbb{P}(V_1\otimes \cdots \otimes V_k)$. In \cite{bb1} we proved that the rank of an element $[p]\in\tau( S_{1, \ldots , 1}(V_1, \ldots , V_k))\setminus S_{1, \ldots , 1}(V_1, \ldots , V_k) $ is $k$ if $[p]$ is not contained in any smaller Segre variety (i.e. with less factors).

\medskip

Before entering the details of the proof of Theorem  \ref{i1} we need the following lemma (Concision or Autarky for multi-homogeneous polynomials or partially symmetric tensors) (see \cite[Proposition 3.1.3.1]{l} for tensors and \cite[Exercise 3.2.2.2]{l} for homogeneous polynomials or symmetric tensors). This lemma will assure that the rank of any $p\in S^{d_1}V_1\otimes \cdots \otimes S^{d_k}V_k$ won't depend on the dimension of the $V_i$'s for $i=1, \ldots, k$.

\begin{definition} Let $W_i\subseteq V_i$ be any non trivial vector subspace for $i=1, \ldots, k$ and assume that 
$p\in S^{d_1}W_1\otimes \cdots \otimes S^{d_k}W_k \subset S^{d_1}V_1\otimes \cdots \otimes S^{d_k}V_k$.

The \emph{rank of $p$ as an element of} $S^{d_1}W_1\otimes \cdots \otimes S^{d_k}W_k$ is 
the minimum integer $r$ such that $[p]\in \langle [p_1], \ldots , [p_{r}]\rangle$ with $[p_i]\in S_{d_1, \ldots , d_k}(W_1, \ldots , W_k) $ for $i=1, \ldots , r$.

The \emph{rank  of $p$  as an element of} $S^{d_1}V_1\otimes \cdots \otimes S^{d_k}V_k$ is 
the minimum integer $r'$ such that $[p]\in \langle [p_1], \ldots , [p_r']\rangle$ with $[p_i]\in \SV$ for $i=1, \ldots , r'$.
\end{definition}

\begin{lemma}[Concision/Autarky]\label{aa1}
Let $W_i\subseteq V_i$ be any non trivial vector subspace for $i=1, \ldots, k$.
The rank $r$  of an element  $p\in S^{d_1}W_1\otimes \cdots \otimes S^{d_k}W_k \subset S^{d_1}V_1\otimes \cdots \otimes S^{d_k}V_k$ as an element of $S^{d_1}V_1\otimes \cdots \otimes S^{d_k}V_k$ is  the same as the rank of $p$ as an element of $S^{d_1}W_1\otimes \cdots \otimes S^{d_k}W_k$. 

For each linear form $l_{i,j}\in V_i^*$ 
such that the multi-homogeneous polynomial $p$ can be written as $p = \sum _{j=1}^{r} \lambda_j l_{1,j}^{ d_1} \cdots  l_{k,j}^{d_k}$, we have $l_{i,j}\in W_i^*$ for all $i=1, \dots , k$, $\lambda_j\in K$, $  j=1, \ldots , r$.

In terms of partially symmetric tensors, this can be rephrased as follows. For each $p_{i,j}\in V_i$ 
such
that $p = \sum _{j=1}^{r} \lambda_j p_{1,j}^{\otimes d_1}\otimes \cdots \otimes p_{k,j}^{\otimes d_k}$ we have $p_{i,j}\in W_i$ $\lambda_j\in K$ for all $i=1, \dots , k$, $  j=1, \ldots , r$.
\end{lemma}

\begin{proof}
Obviously the rank of $p$ as an element of $S^{d_1}W_1^*\otimes \cdots \otimes S^{d_k}W_k^*$ is at least its rank, $r$, as an element of
$S^{d_1}V_1^*\otimes \cdots \otimes S^{d_k}V_k^*$.
To check the opposite inequality and the last assertion of the lemma we first reduce to the case in which $W_i=V_j$  except for one index, say $j=1$, and then to the case in which $W_1$ is a hyperplane of $V_1$ (then one has simply to iterate several times
the construction with $W_i$ a hyperplane of $V_i$ and $W_j=V_j$ for all $j\ne i$).

Let $l_{i,j}\in V_i^*$, $1\le i\le k$, $1\le j\le r$, be  such
that  the decomposition $p = \sum _{j=1}^{r}\lambda_j  l_{1,j}^{ d_1} \cdots l_{k,j}^{ d_k}$ is minimal with $\lambda_j \in K$. Choose homogeneous coordinates 
$V_1^*=K[x_{1,0},\dots ,x_{1,n_1}]$ such
that $W_1= \{x_{1,0}=0\}$. The polynomials $l_{i,j}\in S^{d_i}V_i^*$ are homogenous so $p$ can be written also as $p=\sum_{i=1}^r\lambda_i l_{1,j}^{d_1}\cdots l_{k,j}^{ d_k}$ where $l_{i,j}$ are linear forms in the variables $\{x_{i,0}, \ldots , x_{i,n_i}\}$ and $\lambda_i\in K$ for  $i=1, \ldots , k$, $j=1, \ldots , r$. Let $l_{1,j}=a_jx_{1,0} +l_j(x_{1,1},\dots ,x_{1,n_1})$ be a linear form such that  $a_j\in K$ and $l_j(x_{1,1},\dots ,x_{1,n_1})$ is a linear form in $\{x_{1,1},\dots ,x_{1,n_1}\}$, for $j=1, \ldots ,r $, so
\begin{equation}\label{eqoo1}
p = \sum _{j=1}^{r} \lambda_j (a_jx_{1,0} +l_j(x_{1,1},\dots ,x_{1,n_1}))^{d_1}l_{2,j}^{d_2}\cdots  l_{k,j}^{d_k}.
\end{equation}
Assume now that the lemma is false for $p$, i.e.
assume $a_j\ne 0$ for some $j$, say $a_1 \ne 0$. Since by hypothesis $p\in S^{d_1}W_1^*\otimes S^{d_2}V_2^* \otimes \cdots \otimes S^{d_k}V_k^*$ and since  $W_1 = \{x_{1,0}=0\}$, then  $p$ does not depend on $x_{1,0}$, hence we may substitute $x_{1,0}$ with any linear form
in $x_{1,1},\dots ,x_{1,n_1}$ in (\ref{eqoo1}) and still get an equality.
Setting $x_{1,0}:= -l_1(x_{1,1},\dots ,x_{1,n_1})/a_1$ in (\ref{eqoo1}) we see that $p$ has rank at most $r-1$, that contradicts the minimality of the decomposition of $p$.\end{proof}

The following analysis is quite standard, anyway one can refer for example to \cite{bl}. Since any two points of a projective space are linearly independent, for each $[p]\in\sigma _2(\SV)\setminus \SV$ there is a degree $2$ zero-dimensional
scheme 
$$\Gamma \subset \SV \hbox{ such that } [p]\in \langle \Gamma\rangle.$$
If $[p]\in \sigma_2(\SV)\setminus \tau(\SV)$ then, $\Gamma$ is a smooth scheme (i.e. it has support on two distinct points).
\\
If $[p]\in \tau(\SV)\setminus \SV$ then, $\Gamma$ is a non reduced scheme of degree 2 (i.e. it has support on only one point, such schemes are sometimes called $2$-\emph{jets}).

Now denote 
$$\nu _{d_1,\dots ,d_k}: \PP(V_1)\times \cdots \times \PP(V_k) \to \PP(S^{d_1}V_1 \otimes \cdots \otimes S^{d_k}V_k)$$
the Segre-Veronese embedding of multi-degree $(d_1,\dots ,d_k)$ 
 induced by the complete linear system $|\Oo _{ \PP(V_1)\times \cdots \times \PP(V_k)}(d_1,\dots ,d_k)|$. \\
Hence for
any $[p]\in \tau (\SV)\setminus \SV$ there is a degree 2 zero-dimensional scheme $W_p\subset \mathbb{P}(V_1)\times \cdots \times \PP(V_k)$ with support at only one point such that 
\begin{equation}\label{WT}[p]\in \langle \nu _{d_1,\dots ,d_k}(W_p)\rangle.\end{equation}
 This proof
works for the tangential variety of any smooth manifold embedded in a projective space.
See \cite[Remarks 1 and 2]{bb1} for the uniqueness of $W_p$ and the definition of the following set $I_p\subseteq \{1,\dots ,k\}$.
\begin{notation}\label{notation}
For any $[p]\in \tau (\SV)\setminus \SV$ let $I_p\subseteq \{1,\dots ,k\}$ be the minimal subset such that the scheme 
$W_p$ of (\ref{WT}) depends only  on these factors. 
\end{notation}

We can now prove Theorem \ref{i1}.

\begin{proof}[Proof of Theorem \ref{i1}]

We have to prove only that 
 $r_{d_1, \ldots , d_k}({p}) \leq \sum _{i\in I_p} d_i$ where $I_p$ is as in Notation \ref{notation}. In fact the other inequality is obvious, but let us spend few words to clarify this fact. 

\medskip
 
Let $S^{d_1}V_1\otimes \cdots \otimes S^{d_k}V_k$ be the minimal space containing $p$. So $p\in S^{d_1}V_1\otimes \cdots \otimes S^{d_k}V_k\subset \underbrace{V_1\otimes \cdots \otimes V_1}_{d_1}\otimes \cdots \otimes \underbrace{V_k\otimes \cdots \otimes V_k}_{d_k}$. Therefore, our $[p]\in\tau(\SV)$ can be decomposed  both as $p=\sum_{i=1}^r \lambda_i p_i$ with $[p_i]\in \SV$, $\lambda_i\in K$ and as $p=\sum_{i=1}^{r'}\gamma_iq_i$, $\gamma_i\in K$, where the $[q_i]$'s are elements of the Segre-Veronese variety  $S_{1, \ldots , 1}(\underbrace{V_1, \ldots , V_1}_{d_1}, \ldots, \underbrace{V_k, \ldots , V_k}_{d_k})\subset \underbrace{V_1\otimes \cdots \otimes V_1}_{d_1}\otimes \cdots \otimes \underbrace{V_k\otimes \cdots \otimes V_k}_{d_k}$. 
\\
Now, by \cite{bb1}, $r'=\underbrace{(1+ \cdots + 1)}_{d_1}+\cdots+\underbrace{(1+ \cdots + 1)}_{d_k} =d_1+ \cdots + d_k$. But clearly $r'\leq r$ since $\SV\subset S_{1, \ldots , 1}(\underbrace{V_1, \ldots , V_1}_{d_1}, \ldots, \underbrace{V_k, \ldots , V_k}_{d_k})$.

\medskip

Therefore, let us show that $r_{d_1, \ldots , d_k}({p}) \leq \sum _{i\in I_p} d_i$.

\smallskip 

Let $W_p\subset {\PP(V_1) \times \cdots \times \PP(V_k)}$ be a degree $2$ connected zero-dimensional scheme such that $[p]\in \langle \nu_{d_1, \ldots , d_k} (W_p)\rangle$ as in (\ref{WT}). 

As in \cite{bb1}
 by autarky (Lemma \ref{aa1}) we reduce to the case $I_p = \{1,\dots ,k\}$ (we also need the case $k=1$ proved in \cite[Theorem 32]{bgi} and  the case
$n_i=1$ for all $i=1, \ldots , k$ proved in \cite{bb1}).

Since $I_p= \{1,\dots ,k\}$, we claim that there is a smooth rational curve $C\subset {\PP (S^{d_1}V_1 \otimes \cdots \otimes S^{d_k} V_k)}$ of multi-degree $(d_1,\dots ,d_k)$ such that  $\nu_{d_1, \ldots , d_k}(W_p)\subset C$ (when $k\ge 2$ the curve $C$ is not unique). 
As remarked above, $W_p$ is a 2-jet in the Zariski tangent space of $ {\PP(V_1) \times \cdots \times \PP(V_k)}$ at its support $Supp(W_p)$.
The variety $\PP(V_1) \times \cdots \times \PP(V_k)$ is a compactification of the affine space $\AA ^{n_1+\cdots +n_k}$. Hence there is a map $f: \PP^1\to \PP(V_1) \times \cdots \times \PP(V_k)$ such that, if we fix a point $[q]\in \PP^1$, then $f([q]) =Supp(W_p)$, $W_p$ is the image of the degree $2$ scheme $2q$ of $\PP^1$ and, if $\pi_i$ is the projection of $\PP(V_1) \times \cdots \times \PP(V_k)$ to the $i$-th factor, the maps $\pi _i\circ f$ are either constant or an isomorphism (proof: the intersection of $f(\PP^1)$
with the affine space $\AA ^{n_1+\cdots +n_k}$ is the line through $Supp(W_p)$ spanned by $W_p$). Since $I_p =\{1,\dots ,k\}$, this map has multidegree $(1,\dots ,1)$, i.e. for all $i=1,\dots ,k$, the map $\pi _i\circ f: \PP^1\to \PP^1$ is the isomorphism induced by $|\Oo _{\PP^1}(1)|$. Since $\pi _1\circ f$ is an isomorphism, $f$ is an embedding. Now $\nu_{d_1, \ldots , d_k}(f(\PP^1))$   is our  curve $C\subset \mathbb{P}(S^{d_1}V_1 \otimes \cdots \otimes S^{d_k} V_k)$.
Since $\nu_{d_1, \ldots , d_k}(W_p)\subset C$, we have that $[p]\in \langle \nu_{d_1, \ldots , d_k} (C)\rangle$. 

\begin{notation} Let  as above $C\subset {\PP (S^{d_1}V_1 \otimes \cdots \otimes S^{d_k} V_k)}$ be, as above, a smooth rational curve of multi-degree $(d_1,\dots ,d_k)$. We indicate with $r_C(p)$ the minimum integer $r$ for which there exist $r$ points $[p_1], \ldots , [p_r]\in C$ such that $[p]\in \langle [p_1], \ldots , [p_r] \rangle$ and we call it the \emph{$C$-rank} of $p$.
\end{notation}

Since $C$ is a rational normal curve of degree $d_1+\cdots +d_k$ in its linear
span, we have 
\begin{equation}\label{done}r_{d_1, \ldots , d_k}({p}) \le r_{C}({p}) \le d_1+\cdots +d_s.\end{equation}
The latter inequality is a consequence  of a celebrated  theorem of Sylvester  (see \cite{bgi, cs} for modern and simplified proofs of the same) that can be interpreted as follows: 
\begin{quote}
If $C\subset \PP^n$ is a rational normal curve of degree $d$ and $Z\subset C$ is a minimal zero-dimensional scheme of length $r$ such that a point $[p]\in \langle Z \rangle$, then $[p]$ can be written as a linear combination of $r$ or of $d-r+1$ points on $C$ according with the fact that $Z$ is reduced or not.
\end{quote}

The inequality (\ref{done}) concludes the proof since, as we said at the beginning of the proof, the other inequality is obvious.
\end{proof}


\section{Decomposition of the elements on the tangential variety of a Segre-Veronese variety of two factors}\label{SectionTangential}

We go back to the Segre-Veronese variety of two factors as in Section \ref{SectionTwoFactors} and we keep considering its tangential variety as in Section \ref{RankTg}. After having proved in Section \ref{SectionTwoFactors} how the decomposition of certain bi-homogeneous polynomials (partially symmetric tensors of two factors) has to be done (under certain conditions on the rank and on the degree), and after having computed the rank of the elements in the tangential variety of any Segre-Veronese variety in Section \ref{RankTg}, let us describe how the decompositions of elements in $\tau(\SVd)$ should be done. This will be the content of Theorem \ref{i2} and the purpose of this section will be to prove it.

\begin{notation} A curve $C\subset \PP^{n_1}\times \PP^{n_2}$ is said to have \emph{bi-degree} $(a,b)$  if $\deg (\Oo _C(1,0))=a$ and $\deg (\Oo _C(0,1)) =b$. If such a curve $C$ will have bi-degree $(a,0)$ we will call it an \emph{$\alpha$-curve} of degree $a$ (as in Definition \ref{DefAlphaBeta}, if $a=1$ then $C$ we be called an $\alpha$-line). If $C\subset \PP^{n_1}\times \PP^{n_2}$ will be a curve of bi-degree $(0,b)$ we will call it a \emph{$\beta$-curve} of degree $b$ (as in Definition \ref{DefAlphaBeta}, if $b=1$ then $C$ we be called a $\beta$-line).
\end{notation}

\begin{notation}\label{o} Remind that in (\ref{WT}) we have defined a scheme  $W_p\subset \mathbb{P}(V_1)\times \mathbb{P}(V_2)$ to be the degree 2 zero-dimensional scheme such that the fixed point $[p]\in \tau(\SVd)$ will be contained in $\langle \nu_{d_1,d_2}(W_p)\rangle$. Let here $[o]\in \mathbb{P}(V_1)\times \mathbb{P}(V_2)$ be the support of such a $W_p$.
\end{notation}

\begin{notation}\label{rankG} Let $G$ be a bidegree $(1,1)$ curve
(resp. an $\alpha$-line or a $\beta$-line)  and let $[p]\in \langle \nu_{d_1, d_1}(G)\rangle$. We indicate with $r_{\nu_{d_1, d_2}(G)}(p)$ the minimum $r$ such that $[p]\in \langle [p_1], \ldots , [p_r] \rangle$ with $[p_i]\in \nu_{d_1,d_2}(G)$ for $i=1, \ldots , r$.
\end{notation}

\begin{theorem}\label{i2} 
Take 
$[p]\in \tau(\SVd)$ such that the set $I_p$ defined in Notation \ref{notation} is $I_p  = \{1,2\}$ (resp. $I_p = \{1\}$, resp. $I_p = \{2\}$) and let $W_p\subset \mathbb{P}(V_1)\times \mathbb{P}(V_2)$ and $[o]\in  \mathbb{P}(V_1)\times \mathbb{P}(V_2)$ defined as in Notation \ref{o}.

\begin{enumerate}[(i)]
\item\label{itemi} Let $S$ be one of the schemes computing the rank of $p$, i.e. $S\in \mathcal {S}(p)$ (where $\mathcal {S}(p)$ is defined in Definition \ref{SP}).
Then  $[o]\notin S$ and $S$ is contained in one of the curves $G$ of bidegree $(1,1)$ 
(resp. the unique $\alpha$-line, resp. the unique $\beta$-line) 
containing the unique tangent vector $W_p$. 
If $I_p= \{1,2\}$ and $G$ is not smooth, then $G =L\cup R$
with $L$ a $\beta$-line and $R$ an $\alpha$-line
such that $\{[o]\} = L\cap R$, $\sharp (S\cap L) =d_2$ and $\sharp (S\cap R)=d_1$.

\item\label{itemii} Take any  curve $G$ of bidegree $(1,1)$
(resp. the unique $\alpha$-line, resp. the unique $\beta$-line) 
containing the unique tangent vector $W_p$. We have $r({p}) = r_{\nu_{d_1,d_2} (G)}({p})$ and hence $S\in \mathcal {S}({p})$ for
every  $S\subset G$ with $\sharp (S) =r_{\nu_{d_1,d_2} (G)}({p})$ and
$[p]\in \langle \nu_{d_1,d_2} (S)\rangle$. 
\end{enumerate}
\end{theorem}

\begin{lemma}\label{usa3.00}
Take $Z$ as in Lemma \ref{1.7.1} with $\deg (Z)\le a_1+a_2+1$ and assume the existence of $T\in |\Oo(1,0)|$ such that $\deg (T\cap Z)\ge a_2+2$.
Then there is no $D\in |\Oo (0,1)|$ with $\deg (Z\cap D)\ge a_1+1$.
\end{lemma}

\begin{proof}
If such a $D$ exists, since $\deg (D\cap T) =1$, then $a_1+a_2+1 \ge \deg (Z)\cap \deg (Z\cap (T\cup D)) \ge \deg (Z\cap T)+\deg (Z\cap D)-1 =a_1+a_2+2$,
that is a  contradiction.
\end{proof}

The following lemma can be stated for $\SV$ the Segre-Veronese variety  of any number of factors.

\begin{lemma}\label{a2}
Fix a divisor $D\in |\Oo _{\mathbb{P}^{n_1}\times \cdots \times \mathbb{P}^{n_k}}(b_1,\dots ,b_k)|$ be an effective divisor with $b_i\le d_i$ for all $i=1, \ldots , k$. Fix $[p]\in \PP(S^{d_1}V_1\times \cdots \times S^{d_k}V_k)$. Let $A
\subset \mathbb{P}^{n_1}\times \cdots \times \mathbb{P}^{n_k}$ be 
a
zero-dimensional schemes computing the rank $r_{d_1, \ldots , d_k}(p)$ of $p$, and let $B\subset \mathbb{P}^{n_1}\times \cdots \times \mathbb{P}^{n_k}$, $B\neq A$
be a finite set such that $[p]\in \langle \nu_{d_1, \ldots , d_k} (B)\rangle$ and $[p]\notin \langle \nu_{d_1 , \ldots , d_k} (B')\rangle$ for any $B'\subsetneq B$. 
Then take $Z:= A\cup B$. If $h^1\left(\Ii _{\mathrm{Res}_D(Z)}(d_1-b_1,\dots ,d_k-b_k)\right) =0$,
then  every connected component of $A$ not contained in $D$ is reduced and $A\setminus
A\cap D = B\setminus B\cap D$.
\end{lemma}

\begin{proof}
The proof is completely analogous to the one of \cite[Lemma 5.1]{bb2}.
\end{proof}

We can finally prove Theorem \ref{i2}.

\begin{proof}[Proof of Theorem \ref{i2}:] If $I_p\ne \{1,2\}$, then by autarky for partially symmetric tensors (Lemma \ref{aa1}) we reduce to the case $k=1$ proved in \cite[Theorem 2]{b}.

\smallskip

Therefore consider the case in which $I_p =\{1,2\}$. By autarky for partially symmetric tensors (Lemma \ref{aa1}) we reduce to the case $n_1=n_2=1$.
Take $S\in \mathcal {S}({p})$ and set $Z:= W_p\cup S$. By \cite[Lemma 1]{bb} we have $h^1(\Ii _Z(d_1,d_2)) >0$. Moreover $\deg (Z) \le 2+d_1+d_2$ and equality holds if and
 only if $[o]\notin S$. Take the set-up of Lemma \ref{1.7.1}. First assume the existence of $T\in |\Oo(1,0)|$ such that $\deg (T\cap Z)\ge a_2+2$ and hence
  $\deg (\Res _T(Z))
 \le 2+d_1+d_2-2-d_2 =d_1$. Lemma \ref{a2} gives $h^1\left(\Ii _{\Res _T(Z)}(d_1-1,d_2)\right) >0$, because no connected component of $W_p$ is reduced.
 By Lemma \ref{1.7.1} for the integer $(a_1,a_2) =(d_1-1,d_2)$, there is $T'\in |\Oo (0,1)|$ such that $\deg (T'\cap \Res _T(Z)) \ge d_1+1$. Hence $\deg (Z)\ge \deg ((T+T')\cap Z)
 \ge d_1+d_2+3$, a contradiction. 
 \\
 In the same way we exclude the existence of  $D\in | \Oo (0,1)|$ such that $\deg (D\cap Z)\ge a_1+2$.
 Hence $\deg (Z)=2+d_1+d_2$ (i.e. $o\notin S$) and there is there is $C\in |\Oo (1,1)|$ such that $Z\subset C$.
 
\smallskip

Now we check the last statement of Theorem \ref{i2}. Fix $G\in |\Oo (1,1)|$ such that $W_p\subset G$ and hence $[p]\in \langle \nu_{d_1,d_2} (G)\rangle$.  
The set $\nu_{d_1,d_2} (G)$ is a connected and reduced algebraic set spanning a projective space of dimension $d_1+d_2$. Since $\SVd \supset \nu_{d_1,d_2} (G)$, it is sufficient to prove that $r({p}) \ge  r_{\nu_{d_1,d_2}G}({p})$ where $ r_{\nu_{d_1,d_2}G}({p})$ is defined as in Notation \ref{rankG}. By \cite[Proposition 5.1]{lt} (which is true  even for non-irreducible variety, but reduced and connected schemes) we have  $r_{\nu_{d_1,d_2} (G)}({p}) \le d_1+d_2$. Hence $S\in \mathcal {S}({p})$ for every $S\subset G$ such that $\nu_{d_1, d_2} (S)$ evinces $r_{\nu_{d_1,d_2} (G)}({p})$.

\smallskip

In order to conclude, we need to check second part of (\ref{itemi}) in the case in which $G$ is reducible.

\smallskip

 \quad {\emph {Claim 1:}} Fix $G\in |\Oo _{\mathbb{P}^{n_1}\times \mathbb{P}^{n_2}}(1,1)|$ with $W_p\subset G$ and $G = L\cup R$ with $L\in |\Oo (1,0)|$ and $R\in | \Oo (0,1)|$. Fix $S\subset G$
 such that $S\in \mathcal {S}({p})$. Then $\{[o]\}:= R\cap L$,
$\sharp (S\cap L) =d_2$ and $\sharp (S\cap R)=d_1$.

\smallskip

\quad {\emph {Proof of Claim 1:}} We have $\{[o]\} =R\cap L$, because we are in the case $I_p =\{1,2\}$ and hence neither $W_p\subset L$ nor $W_p\subset R$.
We proved that $[o]\notin S$ and hence $Z:= S\cup W_p$ has degree $d_1+d_2+2$. 
\\
We excluded the existence of  $T\in |\Oo(1,0)|$ such that $\deg (T\cap Z)\ge a_2+2$
and hence $\deg (L\cap Z) \le a_2+1$. 
\\
We excluded the existence of  $D\in |\Oo (0,1)|$ such that $\deg (D\cap Z)\ge a_1+2$ and so $\deg (R\cap Z) \le a_1+1$.
\\
Since $d_1+d_2+2 = \deg (Z)\ge \deg (Z\cap L)+\deg (Z\cap R)$, we get $\deg (Z\cap L) =a_2+1$ and $\deg (Z\cap R)=a_1+1$. Since $\deg (W_p\cap L)=\deg (W_p\cap R)=1$
and $W\cap S=\emptyset$, we get $\deg (S\cap L) =a_2$ and $\deg (S\cap R) =a_1$.
\end{proof}

\section*{Acknowledgements} We like to thank the anonymous referees who urged us to strongly improve the results of this paper.


\begin{thebibliography}{99}

\bibitem{amr} E. S. Allman, C. Matias and J. A. Rhodes, 
Identifiability of parameters in latent structure models with many observed variables, 
Ann. Statist., 37 (2009) 3099--3132.

\bibitem{b1} E. Ballico, 
On the weak non-defectivity of Veronese embeddings of projective spaces, 
Central Eur. J. Math., 3 (2005)  183--187.

\bibitem{b} E. Ballico, 
Subset of the variety $X\subset \PP^n$ evincing the $X$-rank of a point of $\PP^n$, 
Houston J. Math. (to appear).

\bibitem{bb} E. Ballico and A. Bernardi, 
Decomposition of homogeneous polynomials with low rank, 
Math. Z.,  271 (2012) 1141--1149.

\bibitem{bb1}  E. Ballico and A. Bernardi, 
Tensor ranks on tangent developable of Segre varieties,  
Linear Multilinear Algebra, 61 (2013) 881--894.

\bibitem{bb2} E. Ballico and A.  Bernardi, 
Stratification of the fourth secant variety of Veronese variety via the symmetric rank, 
Adv. Pure Appl. Math., 4 (2013) 215--250.

\bibitem{bb3} E. Ballico and A. Bernardi, 
Unique decomposition for a polynomial of low rank, 
Ann. Polon. Math., 108 (2013) 219--224. 

\bibitem{bbcc} E. Ballico, A. Bernardi, M.V. Catalisano and L. Chiantini, 
Grassman secants, identifiability, and linear systems of tensors, 
Linear Algebra Appl., 438 (2013) 121--135.

\bibitem{bc} E. Ballico and L. Chiantini, 
A criterion for detecting the identifiability of symmetric tensors of size three, 
Diff. Geom. Appl., 30 (2012) 233--237.

\bibitem{bgi} A. Bernardi, A. Gimigliano and M. Id\`{a}, 
Computing symmetric rank for symmetric tensors, 
J. Symbolic. Comput.,  46 (2011) 34--53.

\bibitem{bcv} A. Bhaskara, M. Charikar and A. Vijayaraghavan, 
Uniqueness of tensor decompositions with applications to polynomial identifiability, 
Preprint: arXiv:1304.8087 (2013) 1--51.

\bibitem{BC}C. Bocci and L. Chiantini, 
On the identifiability of binary Segre products, 
J. Algebraic Geometry, 22 (2013) 1--11.

\bibitem{bcmt} J. Brachat, P. Comon, B. Mourrain and  E. P. Tsigaridas, 
Symmetric Tensor Decomposition, 
Linear Algebra Appl.,  433 (2010) 1851--1872.

\bibitem{bl} J. Buczy\'nski and J.M. Landsberg, 
On the third secant variety, 
J. Algebraic Combin., 40 (2014) 475--502.

\bibitem{bgl} J. Buczy\'nski, A. Ginensky and J.M. Landsberg,
Determinantal equations for secant varieties and the Eisenbud--Koh--Stillman conjecture,
J. London Math. Soc. 88 (2013) 1--24.

\bibitem{co}  L. Chiantini and G. Ottaviani,
On generic identifiability of 3-tensors of small rank,
SIAM J. Matrix Anal. Applic., 33 (2012) 1018--1037. 

\bibitem{cov}  L. Chiantini, G. Ottaviani and N.Vanniuwenhoven,
An algorithm for generic and low-rank specific identifiability of complex tensors,
SIAM J. Matrix Anal. Applic., 35 (2014) 1265--1287. 

\bibitem{cs} G. Comas and M. Seiguer,
On the Rank of a Binary form,
Found. Comput. Math., 11 (2011)  65--78.

\bibitem{ddl1} I. Domanov and L. De Lathauwer, 
On the uniqueness of the canonical polyadic decomposition of third-order tensors--part I: Basic results and unique- ness of one factor matrix, 
SIAM J. Matrix Anal. Appl., 34 (2013) 855--875.

\bibitem{ddl2} I. Domanov and L. De Lathauwer, 
On the uniqueness of the canonical polyadic decomposition of third-order tensors--part II: Uniqueness of the overall decomposition, 
SIAM J. Matrix Anal. Appl., 34 (2013) 876--903.

\bibitem{ddl3} I. Domanov and L. De Lathauwer, 
Generic uniqueness conditions for the canonical polyadic decomposition and INDSCAL, 
Preprint: arXiv:1405.6238, (2014).

\bibitem{ik} A. Iarrobino, V. Kanev,
\emph{Power sums, Gorenstein algebras, and determinantal loci},
Lecture Notes in
Mathematics, vol. 1721, Springer-Verlag, Berlin, 1999, Appendix C by 
Iarrobino and Steven L. Kleiman.


\bibitem{l} J.M. Landsberg, 
Tensors: Geometry and Applications Graduate Studies in Mathematics, 
Amer. Math. Soc. Providence, 128  (2012).

\bibitem{lt} J.M. Landsberg and Z. Teitler, 
On the ranks and border ranks of symmetric tensors. 
Found. Comput. Math., {{10}} (2010)  339--366.

\end{thebibliography}
\end{document}